\documentclass[11pt]{amsart}
\usepackage{geometry}                
\usepackage{graphicx}
\usepackage{amssymb}
\usepackage{url}
\usepackage{epstopdf}
\numberwithin{equation}{section}

\newcommand{\E}{\mathbb {E}}
\newcommand{\be}{\begin{equation}}
\newcommand{\ee}{\end{equation}}
\newcommand{\R}{\mathbb R}
\newcommand{\e}{\mathcal {E}}

\renewcommand{\>}{\rangle}
\everymath{\displaystyle}
\urldef\shortcourse\url{http://www.ams.org/meetings/short-courses/short-course-general#random%20fields}
\newtheorem{thm}{Theorem}[section]
\newtheorem{lem}[thm]{Lemma}
\newtheorem{prop}[thm]{Proposition}

\theoremstyle{definition}
\newtheorem{defn}{Definition}[section] 
\newtheorem{ex}{Example}[section]
\theoremstyle{remark}
\newtheorem{rmk}{Remark}[section]

\title{Fractional brownian fields over manifolds}
\author{Zachary A. Gelbaum}
\address{Oregon State University}
\email{gelbaumz@math.oregonstate.edu}
\subjclass[2010]{Primary 60G60, 60G15, 58J35}
\date{\today}                                           
\begin{document}

\begin{abstract}  Extensions of the fractional Brownian fields are constructed over a complete Riemannian manifold.  This construction is carried out for the full range of the Hurst parameter $\alpha\in(0,1)$.  In particular, we establish existence, distributional scaling (self-similiarity), stationarity of the increments, and almost sure H\"{o}lder continuity of sample paths.  Stationary counterparts to these fields are also constructed.
\end{abstract}
\maketitle
\section{Introduction}
The fractional Brownian motions and their stationary counterparts are the basic examples of Gaussian random fields over $\R$ and it is natural to ask what are the corresponding examples when $\R$ is replaced by a manifold.  The first to do so was Paul L\'evy (see \cite{MR0190953}), who extended the standard Brownian motion on $\R$ to the standard Brownian field over $\R^d$, now called L\'evy's Brownian motion.  L\'evy then extended this field to the sphere $\mathbb S^d$.  Since then there have been a number of studies aimed at extending both the Brownian motion and the fractional Brownian motion to other manifolds.  This is a natural step in the theory of Gaussian fields in general as one would like to understand how the structure of the index set determines the kinds of fields that can be defined over it.  The geometric and topological structure of Riemannian manifolds make them a convenient and interesting setting for such a study.  When one extends the fractional Brownian motions from $\R$ to $\R^d$ the resulting fields are called \textit{Euclidean fractional Brownian fields} (some authors prefer \textit{L\'evy fractional Brownian motions}) and our purpose in this article is to construct fields over Riemannian manifolds that generalize the Euclidean fractional Brownian fields.

Much of the interest in the fractional Brownian fields ($fBf$'s) over $\R^d$ stems from their distributional invariance and scaling properties.  In particular, if $\alpha\in (0,1)$ denotes the Hurst index and the corresponding field is denoted by $fBf^\alpha$, the increments of the $fBf^\alpha$ are invariant under rotation and translation and the distribution of the $fBf^\alpha$ scales by a power $c^\alpha$ when $\R^d$ is dilated by $c>0$.  Any extension of the $fBf$'s should possess these properties and also reflect the geometry of the index set in question (for an introduction to Gaussian random fields over manifolds focusing on smooth fields see the recent work of Adler and Taylor \cite{MR2319516} and the lecture notes from the short course given at the 2012 joint meetings of the AMS\footnote[1]{\shortcourse}).

As mentioned above the first attempt to extend L\'evy's Brownian motion, $fBf^\frac 12$, from $\R^d$ to a manifold was by L\'evy himself in \cite{MR0190953}.  There he constructed a field over $\mathbb S^d$ with covariance given by \[d(x,o)+d(y,o)-d(x,y),\]  $d(x,y)$ being the geodesic distance between $x$ and $y$ and $o$ being a fixed origin point on the sphere.  Further progress in this direction was made in the work of Molchan (see e.g$.$ \cite{MR913723}) and Gangolli (see \cite{MR0215331}) where the authors dealt with extensions of L\'evy's Brownian motion to other manifolds including the sphere.

Most recently Istas in \cite{MR2198600} studied fields over certain Riemannian manifolds with covariance given by \be \frac12\left(d(x,o)^{2\alpha}+d(y,o)^{2\alpha}-d(x,y)^{2\alpha}\right)\ee where $d(x,y)$ is the metric of the manifold and $o$ is a chosen point.  In particular Istas showed there that (1.1) defines a Gaussian field over compact rank one symmetric spaces and hyperbolic space $\mathbb H^d$ if and only if $\alpha\in(0,1/2]$.  

A common feature of the above approaches is that they begin by looking for covariances of the form $f(x,o)+f(y,o)-f(x,y)$ for some symmetric function $f$; the idea being that over $\R^d$ $o=0$ and $f(x,y)=\|x-y\|_{\R^d}$.  The issue then is to prove that the function so definined does, in fact, define a covariance, i.e$.$, one must establish positive definiteness.  A necessary and sufficient condition for positive definiteness is that $f$ be of \textit{negative type}, for example one can take the above approach on metric spaces $(X,d)$ with metric of negative type (e.g$.$ \cite{MR2266715, MR1855636}).  In general if $d(x,y)$ is the metric of a Riemannian manifold, establishing that $d(x,y)^{2\alpha}$ is of negative type for some $\alpha\in(0,1)$ is non-trivial and indeed, as in \cite{MR2198600}, it has been shown $d(x,y)^{2\alpha}$ can fail to be of negative type.  Moreover, in all the above work this approach necessitates symmetry assumptions on the underlying manifold.

In the present article we take an essentially different approach inspired by the work of Benassi, Jaffard, and Roux (see \cite{MR1462329} and more recently \cite{MR2053563}).  In particular we extend a characterization of the $fBf^\alpha$ in terms of the Laplacian on $\R^d$ to the Riemannian setting via the Laplace-Beltrami operator and the associated heat kernel.  Using this approach we are able to extend the $fBf^\alpha$ to a variety of both compact and non-compact manifolds without any assumptions regarding symmetry of the manifolds and for the full range of $\alpha\in(0,1)$ (see Theorems 3.2-3.4 below).  

Broadly speaking, in order to build a Gaussian random field over a manifold (or any index set) there are two things we must do:  Determine a covariance function and prove that this covariance determines a probability measure on a suitable space of functions, e.g., some space of continuous functions.  If we build our covariance correctly the resulting field will have the properties we would like, and we will be able to use some theorems from probability to show that we get a good probabilistic model, that is, a well defined random element of an appropriate function space. 

This article is structured as follows: in Section 2 we cover some preliminaries regarding Gaussian random fields and analysis on manifolds, in particular the heat kernel of a Riemannian manifold.  In Section 3.1 we describe the motivation behind our approach and define our candidate covariance functions before we study conditions which ensure these covariances exist for a given manifold in section 3.2.  Section 3.3 deals with almost sure sample path regularity and in 3.4 we establish the appropriate distributional invariance properties.  In Section 4 we construct stationary counterparts to the fields of Section 3 and establish the corresponding distributional and sample path properties.  Section 5 contains some open questions concerning geometry and probability encountered in the course of the article and in the appendix we collect some necessary results concerning sample path regularity of Gaussian fields over manifolds.

\section{Preliminaries }
\subsection{Gaussian Random Fields}
Given a complete probability space $(\Omega,\mathcal F,P)$ and some index set $I$ we call a collection of random variables on $\Omega$, $\{X_i(\omega)\}_{i\in I}$, a \textit{Gaussian random field (GRF) over $I$} if for any finite subset $\{i_k\}_1^n\subset I$ the random vector $(X_{i_k})_1^n$ has a joint normal distribution.  Then for each $\omega\in\Omega$, $X_i(\omega)$ defines a real valued function on $I$ called a \textit{sample path} of the field $\{X_i\}$.  We let $\E$ denote the expectation operator, \[\E[X_i]\equiv\int_\Omega X_i(\omega)\,dP(\omega)\qquad i\in I\] and we call \[\E[(X_s-\E[X_s])(X_t-\E[X_t])=\E[X_sX_t]-\E[X_s]\E[X_t]\qquad s,t\in I\] the \textit{covariance} of $\{X_i\}$.  The covariance of a GRF over $I$ defines a symmetric positive definite function on $I\times I$. 

We say two GRF's are equal in \textit{finite dimensional distribution} or simply \textit{in distribution}, denoted $\stackrel{d}{=}$, if their covariances are equal.  We also say two GRF's defined on the same probability space are \textit{versions} of each other if $P(X_i=Y_i)=1$ for all $i\in I$.  The salient analytical feature of GRF's is that for any set $I$ the collection of all GRF's over $I$ is in one to one correspondence up to equality in distribution with the set of all symmetric, positive definite functions on $I\times I$.  In other words a GRF is uniquely determined in distribution by its covariance and every symmetric positive definite function $K$ on $I\times I$ is the covariance of a GRF over $I$, that is, there exists some complete probability space $(\Omega,\mathcal F,P)$ and a $GRF$ $\{X_i(\omega)\}_I$ where for each $i\in I$ $X_i$ is a random variable on $\Omega$.

 We call a GRF \textit{centered} if $\E[X_i]=0$ $\forall i\in I$ and in this case its covariance is given by $\E[X_tX_s]$, $s,t\in I$.  Throughout this article we will only consider centered GRF's.

\subsubsection{The Euclidean Fractional Brownian Fields}  The standard Brownian motion $B_t$ over $[0,\infty)$ is the centered GRF with covariance \[\E[B_sB_t]=s\wedge t=\frac {|s|+|t|-|t-s|}2.\]  From this one generalizes to obtain the fractional Brownian motion $fBm^\alpha$ for $\alpha\in (0,1)$: \[\E[fBm^\alpha_sfBm^\alpha_t]=\frac {|s|^{2\alpha}+|t|^{2\alpha}-|t-s|^{2\alpha}}2.\]   We then have $B_t=fBm^\frac12$.

One then further generalizes to $\R^d$, obtaining the $fBf^\alpha$ as the centered GRF over $\R^d$ with covariance \[\E[fBf^\alpha_xfBf^\alpha_y]=\|x\|_{\R^d}^{2\alpha}+\|y\|_{\R^d}^{2\alpha}-\|x-y\|_{\R^d}^{2\alpha}\] (note that some authors include the constant factor $1/2$).  We remark here that throughout the article we will make a slight abuse of notation and use $\R^d$ to refer both to the usual vector space and to Euclidean space as a manifold, though we doubt this will cause much confusion as the context will make clear what is meant.

One easily sees that the $fBf^\alpha$ is \textit{self similar} of order $\alpha$, i.e$.$, if $fBf^\alpha_c$ denotes the field rescaled field $\{fBf^\alpha_{cx}\}_{x\in\R^d}$ then \[fBf^\alpha_{c}\stackrel{d}{=}c^\alpha fBf^\alpha\qquad\forall\, c>0,\]  and that it has \textit {stationary (or homogeneous) increments}: \[\E[|fBf^\alpha_x-fBf^\alpha_y|^2]=\|x-y\|^{2\alpha}=\|\iota(x)-\iota(y)\|^{2\alpha}=\E[|fBf^\alpha_{\iota(x)}-fBf^\alpha_{\iota(y)}|^2]\] for any isometry $\iota$ on $\R^d$.  Moreover it is known that there exists a version $X_x$ of the $fBf^\alpha$ such that with probability one the sample paths $X_x(\omega)$ are H\"{o}lder continuous of any order $\gamma<\alpha$ and fail to be H\"{o}lder continuous of any order $\gamma>\alpha$ at every point in $\R^d$ (see \cite{MR611857}).

\subsubsection{White Noise}
The treatment here follows \cite{MR1474726}.  Given a probability space $(\Omega,\mathcal F,P)$ we call a complete subspace $G$ of $L^2(\Omega,\mathcal F,P)$ a \textit{Gaussian Hilbert space} if every element of $G$ is a centered Gaussian random variable.  Note that the inner product $H$ inherits from $L^2(\Omega,\mathcal F,P)$ is then \[\<X,Y\>_G=\E[XY].\]  

Given any (real) Hilbert space $H$ there exists a Gaussian Hilbert space $G$ and a unitary map $W:H\to G$ called $W$ the \textit{isonormal process} or \textit{white noise process} on $H$ (one can also consider complex white noises).  If, as is the case below, $H=L^2(M,\mathcal S,d\mu)$ for some measure space $(M,\mathcal S,d \mu)$ then if $B=\{A\in\mathcal S:\mu(A)<\infty\}$ the map from $B\to G$ given by \[W(A)\equiv W(\chi_A)\] determines a Gaussian random measure on $M$.  The properties of such measures will not be important for us here, but we mention them to motivate the notation for $W:H\to G$, given by\[W(f)=\int_M f(z)\,dW(z),\] which we refer to as a \textit{white noise integral} (this is also commonly called a \textit{stochastic integral}).  Starting from a random measure one can construct the integral $\int_M dW$ in close analogy with classical measure theory.  All that will be important for us is the property \[\<f,g\>_H=\E\left[\int_M f\,dW,\int_M g\,dW\right].\]

Now suppose we have a function $h(x,z): M\to L^2(M,d\mu)$, $x\mapsto h(x,z)\in L^2(M,d\mu(z))$.  We can then define a centered GRF $Y_x$ over $X$ by \[Y_x\stackrel{d}{=}\int_Mh(x,z)\,dW(z).\]  The covariance of $Y_x$ is then given by \[\E[Y_xY_y]=\<h(x,z),h(y,z)\>_{L^2}=\int_Mh(x,z)h(y,z)\,d\mu(z).\]  Note that the last expression on the right is in fact positive definite and symmetric.  In this case we call $h$ the integral kernel of $Y$.

\subsection{Analysis on Manifolds}
In what follows we assume throughout that all Riemannian manifolds are complete and of dimension $d$, with $2\leq d<\infty$.  For a manifold $M$ let $\Delta$ denote the Laplace-Beltrami operator, or simply the Laplacian for short, on $M$.  In any local coordinate system the action of $\Delta$ on $C^\infty(M)$ is given by\[\Delta=\frac1{\sqrt{g}}\sum \partial_j\left(g^{ij}\sqrt{g}\partial_i\right)\] where $(g_{ij})$ is the matrix of the Riemannian metric in these coordinates, $(g^{ij})=(g_{ij})^{-1}$, and $\sqrt{g}=\left(det(g_{ij})\right)^\frac12$.  Because $M$ is complete, $\Delta$ is essentially self adjoint (see e.g$.$ \cite{MR705991}) and so we may consider from now on the unique minimal self-adjoint extension of $\Delta$, which we shall write as $\Delta$ also.  Moreover the spectrum of $\Delta$ is contained in $(-\infty,0]$ (see e.g. \cite{MR705991}).  By the spectral theorem we can define the heat semigroup \[e^{t\Delta}=\int_0^\infty e^{-t\lambda}\,dE_\lambda\] where $dE_\lambda$ is the spectral measure of $-\Delta$.  The action of $e^{t\Delta}$ on $L^2(M,dV_g)$, where $dV_g$ denotes the measure derived from the metric $g$, is given by a kernel $H_t(x,y)$: \[e^{t\Delta}(f)(x)=\int_MH_t(x,y)f(y)\,dV_g(y).\]  $H_t(x,y)$ is called the \textit{heat kernel} of $M$.  It is known that $H_t$ is strictly positive, symmetric, and contained in $C^\infty(M\times M\times (0,\infty))$.  Moreover we have the semigroup property \[\int_MH_t(x,z)H_s(z,y)\,dV_g(z)=H_{t+s}(x,y).\]  As a consequence $H_t$ is positive definite for each $t>0$.  As its name suggests, $H_t(x,y)$ is a fundamental solution to the heat equation on $M\times (0,\infty)$:\[\begin{cases}\left(\frac{\partial}{\partial t}-\Delta_x\right)H_t(x,y)=0\\\lim_{t\downarrow0}\int_MH_t(x,y)f(x)\,dx=f(y)\quad\forall\, f\in C_0(M).

\end{cases}\]

There are various constructions of the heat kernel, that given in \cite{MR615626} being most suited to our purposes.  In particular if we let \[\mathcal \e_t(x,y)\equiv \frac{e^{-\frac{d(x,y)^2}{4t}}}{\sqrt{(4\pi t)^d}}\]  then there is an open neighborhood of the diagonal $U\subset M\times M$ such that on $U$ \be\frac{H_t(x,y)}{\e_t(x,y)}=\Phi(t,x,y)\ee where $\Phi(t,x,y)$ is symmetric in $x$ and $y$, $\Phi\in C^k([0,T]\times U)$ $\forall$ $T>0$ where $k$ can be chosen arbitrarily large (see \cite{MR768584} and \cite{MR0282313}), and \[\lim_{t\to0, \,x\to y}\Phi(t,x,y)=1.\]  In other words, for $x$ and $y$ close $H_t\sim \e_t$ as $t\to0$.  Thus on any manifold heat diffusion behaves locally for small times as in Euclidean space.

If $M$ is compact then we also have the following eigenfunction expansion of $H_t$:\be H_t(x,y)=\sum_{k=0}^\infty e^{-\lambda_kt}\phi_k(x)\phi_k(y)\ee where $0=\lambda_0<\lambda_1\leq...\leq\lambda_k\uparrow\infty$ and $\{\phi_k\}$ are the spectrum and orthonormalized $L^2$ eigenfunctions of $-\Delta$ respectively and where (2.2) converges absolutely and uniformly for each $t>0$ (see \cite{MR768584}).

Following \cite{MR768584} we define a \textit{regular domain} to be an open, connected, relatively compact subset $D$ of a complete Riemannian manifold such that $\partial D\neq\emptyset$ is smooth.  In what follows when we refer to the Laplacian of a regular domain we mean the Dirichlet Laplacian with corresponding the heat kernel (see \cite{MR768584}, Chapter 7).  As in the compact case we have an eigenfunction expansion (2.2), the only difference being that $\lambda_0>0$.  If $(M,g)$ is a regular domain in manifold $(N,g)$ then, as noted in \cite{MR1849187}, (2.1) holds in this setting as well.

Now suppose $M$ is complete and non-compact, $\{D_k\}_1^\infty$ is any increasing exhaustion of $M$ by regular domains, and $H^k_t(x,y)$ denotes the Dirichlet heat kernel of $D_k$.  Then if we extend each $H^k$ to be zero outside $\overline D\times\overline D$, $\{H^k_t(x,y)\}_1^\infty$ forms a pointwise increasing sequence on $M\times M\times (0,\infty)$.  It was shown in \cite{MR711862} that \[\lim_{k\to\infty}H^k_t(x,y)=H_t(x,y)\] where $H_t(x,y)$ is the heat kernel defined above.

\section{The Riesz Fields}

\subsection{Motivation and Definition}  As mentioned in the introduction, our first task is to write down a candidate covariance for our fields.  We could write down all the properties we want our field to have and see if this determines a covariance, however even on $\R ^d$ this is non-trivial and as we shall see below, on a general manifold the properties of the Euclidean fractional Brownian fields described above do not uniquely determine a GRF.  The other strategy is to find a characterization of the Euclidean fields that suggests a generalization to manifolds and then verify that this ansatz does indeed yield a probability measure on a nice function space with the properties we want.  This is the strategy we will follow, and so our first task is to find a suitable characterization of the Euclidean field $fBf^\alpha$.  

In \cite{MR1462329} the authors begin by defining a symbol class of pseudodifferential operators over $\R^d$.  From such an operator $A$ they define a Gaussian random field with covariance given by the integral kernel of $A^{-1}$.  The authors are then able to derive all the important properties of this field from properties of the symbol of the operator $A$.  This approach to constructing and studying GRF's is a natural extension of the classical spectral theory of Gaussian processes on $\R$ and demonstrates of the power of the spectral point of view.

The basic heuristic can be described as follows:  Beginning with an unbounded operator $A$ on some $L^2$ space, define and study the GRF determined by the integral kernel of $A^{-1}$.  So in attempting to extend the $fBf^\alpha$ to a Riemannian manifold, we should first seek an operator $A$ that determines the $fBf^\alpha$ in the manner above.

Our starting point is the well known (e.g$.$ \cite{MR1462329} or \cite {MR893393}) spectral representation of the $fBf^\alpha$, \be fBf^\alpha_x\stackrel{d}{=}C_{d,\alpha}\int_{\R^d}\frac {e^{i\<x,\xi\>}-1}{\|\xi\|^{\frac d2+\alpha}}\,d\widehat{W}(\xi),\ee where $\widehat{W}$ is a complex white noise on $L^2(\R^d,dx)$, $dx$ is Lebesgue measure, and $C_{d,\alpha}$ is a constant.  Examining (3.1) we see that, up to a constant, for $f\in H_{-\left(\frac d4+\frac\alpha2\right)}(\R^d)$ \[\int_{\R^d}\frac {e^{i\<x,\xi\>}-1}{\|\xi\|^{\frac d2+\alpha}}\hat{f}(\xi)\,d\xi=(-\Delta)^{-(\frac d4+\frac\alpha2)}(f)(x)-(-\Delta)^{-(\frac d4+\frac\alpha2)}(f)(0).\]  Thus if we denote this last operator above by $A$ then the $fBf^\alpha$ is the unique (in distribution) GRF with covariance given by the Schwarz kernel of the operator $A^* A$,\[\E[fBf^\alpha_xfBf^\alpha_y]=C\int_{\R^d}\frac{e^{i\<x-y,\xi\>}-e^{i\<x,\xi\>}-e^{i\<y,\xi\>}+1}{\|\xi\|^{d+2\alpha}}\,d\xi.\]  

We now have a characterization that extends immediately to manifolds:  Simply replace the Laplacian on $\R^d$ by the Laplace-Beltrami operator of the manifold in question and determine the kernel of the operator $A^*A$.  Following \cite{MR705991} we arrive at the following definitions:  \begin{defn}For a complete Riemannian manifold $M$ with heat kernel $H_t(x,y)$ define the  \textit{Riesz field} $R^\alpha$ to be the GRF with covariance given by
\be\E[R^\alpha_xR^\alpha_y]\equiv\frac1{\Gamma\left(\frac d2+\alpha\right)}\int_0^\infty t^{\frac d2+\alpha-1}\left(H_t(x,y)-H_t(x,o)-H_t(y,o)+H_t(o,o)\right)\,dt\ee
where $o\in M$ is a fixed ``{origin}" and the \textit{stationary (or homogeneous) Riesz field} $hR^\alpha$ the GRF with covariance

\be\E[hR^\alpha_xhR^\alpha_y]\equiv\frac1{\Gamma\left(\frac d2+\alpha\right)}\int_0^\infty t^{\frac d2+\alpha-1}H_t(x,y)\,dt.\ee
\end{defn}
Because $H_t(x,y)$ is positive definite for each $t>0$ and \begin{align}\notag &H_t(x,y)-H_t(x,o)-H_t(y,o)+H_t(o,o)\\\notag&\quad=\int_M\left(H_{t/2}(x,z)-H_{t/2}(o,z)\right)\left(H_{t/2}(y,z)-H_{t/2}(o,z)\right)dV_g(z),\end{align} each of these expressions is symmetric and positive definite, and thus when the integrals exist each determines a GRF over $M$.  Of course the convergence of the above integrals is by no means obvious and our first task in Section 3.2 will be to determine manifolds for which they do converge.
\begin{rmk}  We will see shortly that if either (3.2) or (3.3) exist for some $\alpha_0\in(0,1)$ then it also exists for any $\alpha\in(0,\alpha_0)$.  We say $R^\alpha$ (resp. $hR^\alpha$) \textit{exists} for all $\alpha\in (0,b)$ if (3.2) (resp. (3.3)) is finite for all $\alpha\in(0,b)$, $b\leq1$, and all $x,y\in M$.\end{rmk}

It turns out (Proposition 3.5) that the Riesz field (3.2) extends the $fBf^\alpha$ and that they agree up to a constant in distribution over $\R^d$.  However we will also see that the stationary Riesz field has some claim to be an extension of the $fBf^\alpha$, for example over negatively curved manifolds, even though it does not exist on $\R^d$.

Now let $W$ denote the white noise over $L^2(M,dV_g)$.  We will show that when they exist the Riesz fields admit the following integral representations: 

\be R^\alpha_x\stackrel{d}{=}\frac1{\Gamma\left(\frac d4+\frac\alpha2\right)}\int_M\int_0^\infty t^{\frac d4+\frac\alpha2-1}\left(H_t(x,z)-H_t(o,z)\right)\,dt\,dW(z) \ee
and

\be hR^\alpha_x\stackrel{d}{=}\frac1{\Gamma\left(\frac d4+\frac\alpha2\right)}\int_M\int_0^\infty t^{\frac d4+\frac\alpha2-1}H_t(x,z)\,dt\,dW(z).\ee

The issue is whether or not the functions appearing in the above are in fact square integrable for each $x\in M$.  Let us consider this in detail, first for $hR^\alpha$:

Letting $h_{hR}(x,z)=\frac 1{\Gamma\left(\frac d4+\frac \alpha2\right)}\int_0^\infty t^{\frac d4+\frac\alpha2-1}H_t(x,z)\,dt$ we have

\begin{align}
\notag\<h_{hR}(x,z)&,h_{hR}(y,z)\>_{L^2}\\\notag&=\int_M\left(\frac1{\Gamma\left(\frac d4+\frac\alpha2\right)}\int_0^\infty t^{\frac d4+\frac\alpha2-1}H_t(x,z)\,dt\right)\\\notag&\quad\quad\times\left(\frac1{\Gamma\left(\frac d4+\frac\alpha2\right)}\int_0^\infty s^{\frac d4+\frac\alpha2-1}H_t(y,z)\,ds\right)dV_g(z)\\\notag&=\int_M\int_0^\infty\int_0^\infty\frac1{\Gamma\left(\frac d4+\frac\alpha2\right)^2}t^{\frac d4+\frac\alpha2-1}s^{\frac d4+\frac\alpha2-1}H_t(x,z)H_s(y,z)\,dt\,ds\,dV_g(z)\\\notag&=\int_0^\infty\int_0^\infty\frac1{\Gamma\left(\frac d4+\frac\alpha2\right)^2}t^{\frac d4+\frac\alpha2-1}s^{\frac d4+\frac\alpha2-1}\int_MH_t(x,z)H_s(y,z)\,dV_g(z)\,dt\,ds\\\notag&=\int_0^\infty\int_0^\infty\frac1{\Gamma\left(\frac d4+\frac\alpha2\right)^2}t^{\frac d4+\frac\alpha2-1}s^{\frac d4+\frac\alpha2-1}H_{t+s}(x,y)\,dt\,ds\\\notag&=\int_0^\infty\int_s^\infty\frac1{\Gamma\left(\frac d4+\frac\alpha2\right)^2}(t-s)^{\frac d4+\frac\alpha2-1}s^{\frac d4+\frac\alpha2-1}H_t(x,y)\,dt\,ds\\\notag&=\int_0^\infty\int_0^t\frac1{\Gamma\left(\frac d4+\frac\alpha2\right)^2}(t-s)^{\frac d4+\frac\alpha2-1}s^{\frac d4+\frac\alpha2-1}\,dsH_t(x,y)\,dt
\end{align}
by the positivity of $H_t(x,y)$ and the semigroup property. 

Next note that if $g(s)=\frac1{\Gamma\left(\frac d4+\frac\alpha2\right)}s^{\frac d4+\frac \alpha2-1}$ then \[\int_0^t\frac1{\Gamma\left(\frac d4+\frac\alpha2\right)^2}(t-s)^{\frac d4+\frac\alpha2-1}s^{\frac d4+\frac\alpha2-1}\,ds=g* g(t)\] where $*$ denotes the finite convolution $f* g(t)\equiv\int_0^tf(t-s)g(s)\,ds$.  If $\mathcal L$ denotes the Laplace transform we have the well known property $\mathcal L\left(f* g\right)=\mathcal L(f)\mathcal L(g)$.  Applying this to $g* g$ above we have \[\mathcal L(g* g)(s)=\left(\mathcal L(g)\right)^2(s)=\left(s^{-(\frac d4+\frac\alpha2)}\right)^2=s^{-(\frac d2+\alpha)}.\]  Then inverting $\mathcal L$ we obtain \[\frac1{\Gamma(\frac d2+\alpha)}t^{\frac d2+\alpha-1}=\mathcal L^{-1}\left(s^{-(\frac d2+\alpha)}\right)=\int_0^t\frac1{\Gamma\left(\frac d4+\frac\alpha2\right)^2}(t-s)^{\frac d4+\frac\alpha2-1}s^{\frac d4+\frac\alpha2-1}\,ds.\]  Substituting this into the integral defining $\<h_{hR}(x,z),h_{hR}(y,z)\>_{L^2}$ above yields \[\frac1{\Gamma(\frac d2+\alpha)}\int_0^\infty t^{\frac d2+\alpha-1}H_t(x,y)\,dt.\]  Thus whenever $hR^\alpha$ exists it is given by (3.5).

Turning now to (3.2), let $h_R(x,z)=\frac 1{\Gamma\left(\frac d4+\frac \alpha2\right)}\int_0^\infty t^{\frac d4+\frac\alpha2-1}\left(H_t(x,z)-H_t(o,z)\right)\,dt$.  Then \begin{align}\notag\|h_R&(x,z)\|^2_{L^2}\\\notag&\leq\int_M\int_0^\infty\int_0^\infty s^{\frac d4+\frac \alpha2-1}t^{\frac d4+\frac \alpha2-1}|H_t(x,z)-H_t(o,z)||H_s(x,z)-H_s(o,z)|\,ds\,dt\,dV_g(z)\\\notag&=\int_0^\infty\int_0^\infty s^{\frac d4+\frac \alpha2-1}t^{\frac d4+\frac \alpha2-1}\int_M|H_t(x,z)-H_t(o,z)||H_s(x,z)-H_s(o,z)|\,dV_g(z)\,ds\,dt\\\notag&\leq\int_0^\infty\int_0^\infty s^{\frac d4+\frac\alpha2-1}t^{\frac d4+\frac \alpha2-1}\|H_t(x,\cdot)-H_t(o,\cdot)\|_2\|H_s(x,\cdot)-H_s(o,\cdot)\|_2\,ds\,dt \\\notag&=\left(\int_0^\infty t^{\frac d4+\frac \alpha2-1}\|H_t(x,\cdot)-H_t(o,\cdot)\|_2\,dt\right)^2\\\notag&=\left(\int_0^\infty t^{\frac d4+\frac \alpha2-1}\sqrt{H_t(x,x)-2H_t(x,o)+H_t(o,o)}\,dt\right)^2.\end{align}

Recall that if $M$ is any Riemannian manifold then from (2.1) for any $x,y\in M$ we have that $H_t(x,y)=O(t^{-\frac d2})$ as $t\to0$.  So then \[\int_0^1 t^{\frac d4+\frac \alpha2-1}\sqrt{H_t(x,x)-2H_t(x,o)+H_t(o,o)}\,dt<\infty\] and \[\int_0^1 t^{\frac d2+\alpha-1}\left(H_t(x,x)-2H_t(x,o)+H_t(o,o)\right)\,dt<\infty\] for all $\alpha\in(0,1)$.

Next notice that if $\alpha+\epsilon<b$ \begin{align}\notag\int_1^\infty t^{\frac d4+\frac \alpha2-1}&\sqrt{H_t(x,x)-2H_t(x,o)+H_t(o,o)}\,dt\\\notag =&\int_1^\infty t^{\frac d4+\frac {\alpha}2+\epsilon-(1+\epsilon)}\sqrt{H_t(x,x)-2H_t(x,o)+H_t(o,o)}\,dt\\\notag\leq&\left(\int_1^\infty t^{-(1+\epsilon)}\,dt\right)^\frac12\left(\int_1^\infty t^{\frac d2+ \alpha+\epsilon-1}\left({H_t(x,x)-2H_t(x,o)+H_t(o,o)}\right)\,dt\right)^\frac12
\end{align} 
by Cauchy-Schwarz.  Thus if $R^\alpha$ exists for all $\alpha\in (0,b)$ we may interchange the order of integration as with $hR^\alpha$ to obtain \begin{align}\notag\<h_R(x,z)&,h_R(y,z)\>_{L^2}\\\notag&=\frac1{\Gamma\left(\frac d2+\alpha\right)}\int_0^\infty t^{\frac d2+\alpha-1}\left(H_t(x,y)-H_t(x,o)-H_t(y,o)+H_t(o,o)\right)\,dt\\\notag&=\E[R^\alpha_xR^\alpha_y]\end{align} 
for all such $\alpha$.

In either case of (3.2) or (3.3) we see that the integrands are continuous on $(0,\infty)$ so by (2.1) convergence depends only on the behavior of the integrand at infinity.  Thus the existence of both $R^\alpha_x$ and $hR^\alpha_x$ will depend on the large-time asymptotics of $H_t(x,y)$.  These depend on the manifold in question and we will treat distinct cases below. 
\subsection{Existence}
\subsubsection{The Compact Case}

We have the following:

\begin{thm} If $M$ is a compact Riemannian manifold, then the Riesz field of order $\alpha$ exists over $M$ for any $\alpha\in(0,1)$ and the stationary Riesz field does not exist over $M$ for any $\alpha\in(0,1)$.

\end{thm}

\begin{proof}  Recall (2.2): \[ H_t(x,y)=\sum_{k=0}^\infty e^{-\lambda_kt}\phi_k(x)\phi_k(y).\]   We have \[H_t(x,x)-2H_t(o,x)+H_t(o,o)=\sum_{k=1}^\infty e^{-\lambda_kt}|\phi_k(x)-\phi_k(o)|^2=O(e^{-\lambda_1t})\qquad \forall\, x\in M\] and $\lambda_1>0$.  Then (3.2) is clearly finite for any $x\in M$ and all $\alpha\in(0,1)$. 

To see that $hR^\alpha_x$ does not exist on $M$ notice that $\lim_{t\to0}H_t(x,y)=\mbox{Vol}(M)^{-1}\neq0$\hfill\\ $\forall$ $x,y\in M$.

\end{proof}

\begin{thm} If $M$ is regular domain then $hR^\alpha$, and thus by linearity $R^\alpha$, exists for any $\alpha\in (0,1)$.\end{thm}
\begin{proof}  As above let \[H_t(x,y)=\sum_{k=0}^\infty e^{-\lambda_kt}\phi_k(x)\phi_k(y).\] Then $\lambda_0>0$ and $H_t(x,y)=O(e^{-\lambda_0t})$ for each $x,y\in M$.

\end{proof}

We note here that in either case above we may integrate term by term using the eigenfunction expansions of $H_t$ to obtain a series expression for the covariance of $R^\alpha$ and $hR^\alpha$ as follows:

For $R^\alpha$ and $M$ compact we have \begin{align}\notag\E[R^\alpha_xR^\alpha_y]&=\frac1{\Gamma\left(\frac d2+\alpha\right)}\int_0^\infty t^{\frac d2+\alpha-1}H_t(x,y)-H_t(x,o)-H_t(y,o)+H_t(o,o)\,dt\\\notag&=\frac1{\Gamma\left(\frac d2+\alpha\right)}\int_0^\infty t^{\frac d2+\alpha-1}\sum_{k=0}^\infty e^{-\lambda_kt}(\phi_k(x)-\phi_k(o))(\phi_k(y)-\phi_k(o))\,dt\\\notag&=\frac1{\Gamma\left(\frac d2+\alpha\right)}\int_0^\infty t^{\frac d2+\alpha-1}\sum_{k=1}^\infty e^{-\lambda_kt}(\phi_k(x)-\phi_k(o))(\phi_k(y)-\phi_k(o))\,dt
\\\notag&\leq\frac1{\Gamma\left(\frac d2+\alpha\right)}\left(\int_0^\infty t^{\frac d2+\alpha-1}\sum_{k=1}^\infty e^{-\lambda_kt}|\phi_k(x)-\phi_k(o)|^2\,dt\right)^\frac12\\\notag&\quad\times\left(\int_0^\infty t^{\frac d2+\alpha-1}\sum_{k=1}^\infty e^{-\lambda_kt}|\phi_k(y)-\phi_k(o)|^2\,dt\right)^\frac12\\\notag&=\left(\E[|R^\alpha_x|^2]\E[|R^\alpha_y|^2]\right)^\frac12, \end{align} which we know from above to be finite.

Then by dominated convergence we may integrate term by term to obtain
\begin{align}\notag\E[R^\alpha_xR^\alpha_y]&=\frac1{\Gamma\left(\frac d2+\alpha\right)}\sum_{k=1}^\infty\frac{\Gamma\left(\frac d2+\alpha\right)}{\lambda_k^{\frac d2+\alpha}}(\phi_k(x)-\phi_k(o))(\phi_k(y)-\phi_k(o))\\\notag&=\sum_{k=1}^\infty(\lambda_k)^{-\left(\frac d2+\alpha\right)}(\phi_k(x)-\phi_k(o))(\phi_k(y)-\phi_k(o)).\end{align}  In particular \[R^\alpha_x\stackrel {d}=\sum_{k=1}^\infty (\lambda_k)^{-\left(\frac d4+\frac\alpha2\right)}(\phi_k(x)-\phi_k(o))\xi_k\] where $\{\xi_k\}$ is an i.i.d$.$ collection of standard normal random variables, the series converging in $L^2(M)$ almost surely.

The same equality holds for $M$ a regular domain if we number the spectrum as $\{\lambda_k\}_1^\infty$.  Similar arguments show that for $M$ a regular domain \[\E[hR^\alpha_xhR^\alpha_y]=\sum_{k=1}^\infty(\lambda_k)^{-\left(\frac d2+\alpha\right)}\phi_k(x)\phi_k(y)\] and \[hR^\alpha_x\stackrel {d}=\sum_{k=1}^\infty (\lambda_k)^{-\left(\frac d4+\frac\alpha2\right)}\phi_k(x)\xi_k.\]
\begin{ex}Let $M=\mathbb S^2$.  Then in terms of the spherical harmonics $\{Y_{km}\}$ we have \[H_t(x,y)=\sum_{k=0}^\infty e^{-k(k+1)t}\sum_{m=-k}^{k}Y_{km}(x)Y_{km}(y).\]  Applying the harmonic addition formula we have \[H_t(x,y)=\sum_{k=0}^\infty e^{-k(k+1)t}\frac{2k+1}{4\pi}P_k(\cos\theta_{xy})\] where $P_k$ is the $k$-th Legendre Polynomial and $\<x,y\>=\cos\theta_{xy}$.  Fixing an origin point $o\in \mathbb S^2$ we then have 
\[\E[R^\alpha_xR^\alpha_y]=\sum_{k=1}^\infty\left(k(k+1)\right)^{-\left(\frac d2+\alpha\right)}\frac{2k+1}{4\pi}\left(P_k(\cos\theta_{xy})-P_k(\cos\theta_{xo})-P_k(\cos\theta_{yo})+P_k(1)\right).\]

\end{ex}
\begin{ex} Let $M=\mathbb D=\{x\in \mathbb R^2:|x|<1\}$ and $J_k$ the Bessel function of the first kind of order $k$, $k=0,1,2..$.  Then if $\lambda_k^1<\lambda_k^2<...$ are the positive zeroes of $J_k$, using polar coordinates on $\mathbb D$ we have \[\E[hR^\alpha_{(r,\theta)}hR^\alpha_{(R,\phi)}]=\frac{\sqrt2}{\pi}\sum_{k,l}\frac{(\lambda_k^l)^{-(d+2\alpha)}}{|J_{k+1}(\lambda_k^l)|}J_k(\lambda_k^lr)J_k(\lambda_k^lR)\left(\cos(k(\theta-\phi))+\sin(k(\theta+\phi))\right).\]

\end{ex}
\subsubsection{The Non-Compact Case}

For the case of $M$ non-compact, first let us show by example that we cannot establish existence in general.

\begin{ex}Let $M=\mathbb S^1\times\R$.  Then we have \[H^M_t((\theta,x),(\phi,y))=H^{\mathbb S}_t(\theta,\phi)H^\R_t(x,y)\] where $H^M$ is the heat kernel of $M$, $H^\mathbb S$ is the heat kernel of $\mathbb S^1$, and $H^\R$ is the usual heat kernel on $\R$ (see \cite{MR2569498}, Theorem 9.11).

We then have that \begin{align}\notag H^M_t((\theta,x),(\theta,x))-2H^M_t((\theta,x),(\phi,y)+H^M_t((\phi,y),(\phi,y))&\sim \frac1{\pi}\frac{1-e^{\frac{-|x-y|^2}{4t}}}{\sqrt{(4\pi t)}}\\\notag&=O(t^\frac 32)\quad\text{as}\,\, t\to\infty\end{align} for any $(\theta,x),(\phi,y)\in M$.  So $\E[|R^\alpha_p|^2]=\infty$ $\forall\, p\in M$ and $\alpha\geq1/2$ and thus $R^\alpha$ does not exist over $M$ for this range of $\alpha$.  Using $\mathbb S^2$ instead in the above we have that $R^\alpha$ fails to exist for all $\alpha\in(0,1)$.

\end{ex}

However, for certain manifolds such that $\mbox{Vol}(M)<\infty$ we have a situation similar to the compact case:

\begin{thm}  Suppose $M$ is non-compact with $Ric(M)\geq-\kappa^2$, $\kappa\in \R$, and $\mbox{Vol}(M)<\infty$.  Let $\overline\lambda(M)=\inf_{\Omega\subset M}\left\{\lambda_1:\sigma(\Omega)=\{\lambda_k\}_{k=0}^\infty\right\}$ where the infimum is taken over regular domains $\Omega\subset M$ and $\sigma(\Omega)$ denotes the Dirichlet spectrum of $\Omega$. Then if $\overline\lambda(M)>0$ $R^\alpha$ exists over $M$ for any $\alpha\in (0,1)$ and $hR^\alpha$ does not.
\end{thm}

\begin{proof}That $hR^\alpha$ does not exist follows from the fact that on such $M$ \[\lim_{t\to\infty}H_t(x,y)=\frac1{\mbox{Vol}(M)}\neq0\qquad\forall \,x,y\in M.\]  For $R^\alpha$, under the hypothesis of the theorem it was shown in \cite{MR1482041} that \[H_t(x,y)-\frac1{\mbox{Vol}(M)}=O\left(e^{-\frac{\overline\lambda(M)}2t}\right)\] and so (3.2) converges $\forall$ $\alpha\in(0,1)$.

\end{proof}

We now turn to our main existence theorem for the Riesz fields over non-compact manifolds followed by some examples.  Below we use the following notation: \[D_p(r)\equiv\{x\in M:d(x,p)< r\}\] and \[V_p(r)\equiv \mbox{Vol}\left(D_p(r)\right)=\int_{D_p(r)}dV_g.\]  We write $H_t=\overline O(t^{-\frac\nu2})$ if there exist two distinct points $x_k\in M$, $k=1,2$, and constants $C_k>0$ such that \[H_t(x_k,x_k)\leq C_kt^{-\frac\nu2}\qquad\forall\, t\geq1.\]  In that case using Theorem 1.1 of \cite{MR1443330} we know that for any $\delta>0$ there exists a constant $C_\delta>0$ such that for all $t\geq1 $ and all $x,y\in M$ \[H_t(x,y)\leq C_\delta t^{-\frac \nu2}e^{-\frac{d(x,y)^2}{(4+\delta)t}}.\]

\begin{thm}  Let $M$ be non-compact. 

(1)  Suppose $Ric(M)\geq0$.  Then $hR^\alpha$ does not exist for any $\alpha\in(0,1)$.  If \[H_t= \overline O\left(t^{-\left(\frac d2-\beta\right)}\right)\] and \[\varlimsup_{r\to\infty}\frac{V_x(r)}{r^{d-2\beta}}<\infty\qquad\forall\, x\in M\] for some $\beta\in[0,1)$ then $R^\alpha$ exists over $M$ for any $\alpha\in (0,1-\beta)$.

(2) Suppose that  \[H_t=\overline O \left(t^{-\left(\frac d2+\beta\right)}\right)\] for some $\beta>0$.  Then $hR^\alpha$ (and thus $R^\alpha$ also) exists for any $\alpha\in(0,\min\{\beta,1\})$.
\end{thm}
\begin{proof}(1):  To begin we note that our hypothesis $H_t= \overline O(t^{-( d/2-\beta)})$ implies the following gradient bound for $H_t$ (see \cite{MR1023321}): For all $x,y\in M$ and $t\geq1$
\be|\nabla_xH_t(x,y)|\leq C^\prime_\delta t^{-\left(\frac d2-\beta+\frac12\right)}e^{-\frac{d(x,y)^2}{(4+\delta)t}}\ee for some constant $C^\prime_\delta>0$.

Recall that by Cauchy-Schwarz in order for for (3.2) to converge it is sufficient to show that \[\int_1^\infty t^{\frac d2+\alpha-1}\left(H_t(x,x)-2H_t(x,o)+H_t(o,o)\right)\,dt<\infty\] for the specified range of $\alpha$.  Moreover, by first restricting to a compact subset $K\subset M$ we may assume positive injectivity radius, i.e., $\exists$ $r>0$ such that $d(x,y)<r$ implies that $x,y$ belong to some normal neighborhood.  By repeated use of the triangle inequality we see that existence for all such $x,y$ implies existence on all of $K$, and since $K$ was arbitrary, on all of $M$.  

To that end let $D=D_p(r)$ be a normal neighborhood containing $x$ and $o$.  We first apply the mean value theorem: \begin{align}\notag\int_1^\infty t^{\frac d2+\alpha-1}&\left(H_t(x,x)-2H_t(x,o)+H_t(o,o)\right)\,dt\\\notag= &\int_1^\infty t^{\frac d2+\alpha-1}\int_M|H_t(x,z)-H_t(o,z)|^2\,dV_g(z)\,dt\\\notag\leq &\,d(x,o)^2\int_1^\infty t^{\frac d2+\alpha-1}\int_M|\nabla_xH_t(\xi_z,z)|^2\,dV_g(z)\,dt\end{align} for some $\xi_z$ lying on some curve (parametrized to have unit velocity) contained in $D_p$ and joining $x$ and $o$.  We now apply (3.6),

\begin{align}\notag\int_1^\infty t^{\frac d2+\alpha-1}&\int_M|\nabla_xH_t(\xi_z,z)|^2\,dV_g(z)\,dt\\\notag\leq&\,C\int_1^\infty t^{-\frac d2+\alpha+2\beta-2}\int_Me^{-\frac{2d(\xi_z,z)^2}{(4+\delta)t}}\,dV_g(z)\,dt.\end{align}

We have \begin{align}\notag\int_1^\infty t^{-\frac d2+\alpha+2\beta-2}&\int_Me^{-\frac{2d(\xi_z,z)^2}{(4+\delta)t}}\,dV_g(z)\,dt\\\notag=&\int_1^\infty t^{-\frac d2+\alpha+2\beta-2}\int_{D}e^{-\frac{2d(\xi_z,z)^2}{(4+\delta)t}}\,dV_g(z)\,dt\\\notag&\quad+\int_1^\infty t^{-\frac d2+\alpha+2\beta-2}\int_{M\backslash D}e^{-\frac{2d(\xi_z,z)^2}{(4+\delta)t}}\,dV_g(z)\,dt\\\notag\leq&\,\mbox{Vol}(D)\int_1^\infty t^{-\frac d2+\alpha+2\beta-2}\,dt\\\notag&\quad+\int_{M\backslash D}\int_0^\infty t^{-\frac d2+\alpha+2\beta-2}e^{-\frac{2d(\xi_z,z)^2}{(4+\delta)t}}\,dt\,dV_g(z).
\end{align}

By hypothesis $\int_1^\infty t^{-\frac d2+\alpha+2\beta-2}\,dt<\infty$ so we only need to show \[\int_{M\backslash D}\int_0^\infty t^{-\frac d2+\alpha+2\beta-2}e^{-\frac{2d(\xi_z,z)^2}{(4+\delta)t}}\,dt\,dV_g(z)<\infty.\]  We have \begin{align}\notag&\int_{M\backslash D}\int_0^\infty t^{-\frac d2+\alpha+2\beta-2}e^{-\frac{2d(\xi_z,z)^2}{(4+\delta)t}}\,dt\,dV_g(z)\\\notag&\qquad=\left(\frac{4+\delta}2\right)^{\frac d2-\alpha-2\beta+1}\Gamma\left(\frac d2-\alpha-2\beta+1\right)\int_{M\backslash D}d(\xi_z,z)^{-d+2\alpha+4\beta-2}\,dV_g(z).\end{align} 

Recall $D=D_p(r)$ and let \[A_k=D_p(r+k)\backslash D_p(r+k-1)\qquad k=1,2,3...\]   By monotone convergence \begin{align}\notag\int_{M\backslash D}d(\xi_z,z)^{-d+2\alpha+4\beta-2}\,dV_g(z)&=\sum_{k=1}^\infty\int_{A_k}d(\xi_z,z)^{-d+2\alpha+4\beta-2}\,dV_g(z)\\\notag&\leq\sum_{k=1}^\infty\frac{\mbox{Vol}(A_k)}{(r+k-1)^{d-2\alpha-4\beta+2}}\\\notag&=\sum_{k=1}^\infty\frac{V_p(r+k)-V_p(r+k-1)}{(r+k-1)^{d-2\alpha-4\beta+2}}.\end{align}

Because $Ric(M)\geq0$ we have (see \cite{MR1103113} or \cite{MR658471}) \[V_p(cr)\leq c^dV_p(r)\qquad\forall\, r>0,\,c\geq1.\]  Thus \begin{align}\notag\sum_{k=1}^\infty\frac{V_p(r+k)-V_p(r+k-1)}{(r+k-1)^{d-2\alpha-4\beta+2}}&\leq \sum_{k=1}^\infty\frac{V_p(r+k-1)\left(\frac{(r+k)^d-(r+k-1)^d}{(r+k-1)^d}\right)}{(r+k-1)^{d-2\alpha-4\beta+2}}\\\notag&\leq C\sum_{k=1}^\infty\frac{(r+k-1)^{d-2\beta}\left(\frac{(r+k)^d-(r+k-1)^d}{(r+k-1)^d}\right)}{(r+k-1)^{d-2\alpha-4\beta+2}}\\\notag&=C\sum_{k=1}^\infty\frac{{(r+k)^d-(r+k-1)^d}}{(r+k-1)^{d-2\alpha-2\beta+2}}\end{align}

The convergence of this last sum is equivalent to that of \[\sum_{k=1}^\infty\frac{k^{d-1}}{k^{d-2\alpha-2\beta+2}}=\sum_{k=1}^\infty k^{2\alpha+2\beta-3}.\]   By hypothesis $\alpha<1-\beta$, which implies \[\sum_{k=1}^\infty k^{2\alpha+2\beta-3}<\sum_{k=1}^\infty k^{-(1+\epsilon)}<\infty\] for some $\epsilon>0$.  

To see that $hR^\alpha$ does not exist on $M$ for any $\alpha$, we note that (see e.g$.$ \cite{MR1103113})\[ Ric(M)\geq0\Rightarrow H_t(x,y)\geq (4\pi t)^{-\frac d2}e^{-\frac{d(x,y)^2}{4t}}\] for all $x,y\in M$ and $t>0$.  Thus \[\int_0^\infty t^{\frac d2+\alpha-1}H_t(x,y)\,dt=\infty\] for all $x,y\in M$ and any $\alpha\in(0,1)$.

To prove (2), simply write \[\int_1^\infty t^{\frac d2+\alpha-1}H_t(x,y)\,dt\leq C\int_1^\infty t^{\alpha-\beta-1}\,dt<\infty.\]

\end{proof}

We are now in a position to show that, over $\R^d$, $R^\alpha$ agrees up to a constant with the $fBf^\alpha$ in distribution.  We could do this abstractly using arguments along the lines of Section 3.1, however we can also make a simple explicit calculation.  Note that $\R^d$ satisfies the first hypothesis of Theorem 3.4 with $\beta=0$.  Thus $R^\alpha$ exists there and if we choose $o=0$ has covariance \[\E[R^\alpha_xR^\alpha_y]=\frac1{\Gamma(\frac d2+\alpha)}\int_0^\infty t^{\frac d2+\alpha-1}(H_t(0,0)-H_t(x,0)-H_t(y,0)+H_t(x,y))\,dt.\]

\begin{prop}  If $M=\R^d$ then $H_t(x,y)=\frac1{\sqrt{(4\pi t)^d}}{e^{\frac{-\|x-y\|^2}{4t}}}$ and for all $x,y\in\mathbb R^d$ and for $\alpha\in(0,1)$ \[\E[R^\alpha_xR^\alpha_y]=C_\alpha\left(\|x\|^{2\alpha}+\|y\|^{2\alpha}-\|x-y\|^{2\alpha}\right)\]
where $C_\alpha$ is the positive constant given by \[C_\alpha=\frac{-\Gamma(-\alpha)}{4^{\frac d2+\alpha}(\pi)^\frac d2\Gamma(\frac d2+\alpha)}.\]  
\end{prop}

\begin{proof}First note that if either $x=0$ or $y=0$ the result is trivial; thus we assume otherwise.  The integral defining $\E[R^\alpha_xR^\alpha_y]$ reduces to \[\frac1{\sqrt{(4\pi)^d}}\int_0^\infty t^{\alpha-1}(1-e^{\frac{-\|x\|^2}{4t}}-e^{\frac{-\|y\|^2}{4t}}+e^{\frac{-\|x-y\|^2}{4t}})\,dt,\] which we recognize as a Mellin transform.  Let $F_a(t)=\chi_{[a,\infty)}(t)-e^{\frac{-\|x\|^2}{4t}}-e^{\frac{-\|y\|^2}{4t}}+e^{\frac{-\|x-y\|^2}{4t}}$ with $a>0$.  Then $F_a(t)=O(t^{-1})$ as $t\to\infty$ and $F_a(t)=o(t^N)$ as $t\to0$ $\forall$ $N>0$.  Thus \[\int_0^\infty t^{s-1}F_a(t)\,dt\] converges absolutely for all $s\in \mathbb C$ with $\Re(s)<1$ and defines an analytic function there. 

On the other hand for $-1<\Re(s)<0$ we have by direct calculation that \[\int_0^\infty t^{s-1}F_a(t)\,dt=\frac{a^s}{s}+\frac{-\|x\|^{2s}-\|y\|^{2s}+\|x-y\|^{2s}}{4^{s}}\Gamma(-s).\]

By analytic continuation this last equality holds for $0<\Re(s)<1$ as well.   For such $s$ we have by dominated convergence \[\int_0^1t^{s-1}F_0(t)\,dt=\lim_{a\to0}\int_0^1 t^{s-1}F_a(t)\,dt.\]  Now for $a<1$  \[\int_1^\infty t^{s-1}F_a(t)\,dt=\int_1^\infty t^{s-1}F_0(t)\,dt\] and so, noting $F_0(t)\geq 0$, we have using dominated convergence \begin{align}\notag\int_0^\infty t^{s-1}F_0(t)\,dt&=\int_0^1t^{s-1}F_0(t)\,dt+\int_1^\infty t^{s-1}F_0(t)\,dt\\\notag&=\left(\lim_{a\to0^+}\int_0^1t^{s-1}F_a(t)\,dt\right)+\int_1^\infty t^{s-1}F_0(t)\,dt\\\notag&=\lim_{a\to0^+}\left(\int_0^1 t^{s-1}F_a(t)\,dt+\int_1^\infty t^{s-1}F_0(t)\,dt\right)\\\notag&=\lim_{a\to0^+}\int_0^\infty t^{s-1}F_a(t)\,dt\\\notag&
=\frac{-\|x\|^{2s}-\|y\|^{2s}+\|x-y\|^{2s}}{4^{s}}\Gamma(-s)\end{align} 
\end{proof}

\begin{ex}Suppose $M$ is non-compact with $Ric(M)\geq0$ and \[\varliminf_{R\to\infty}\frac {V_p(R)}{R^d}=\theta\in(0,1)\] for some $p\in M$ (cf$.$ the Bishop-Gromov comparison theorem).  Then $R^\alpha$ exists over $M$ for any $\alpha\in (0,1)$ and $hR^\alpha$ does not.  Indeed, in \cite{MR847950} it is shown that $H_t(x,y)=O(t^{-\frac d2})$ for every $x,y\in M$.  Theorem 3.4 applies once we note that for all $p\in M$ \[Ric(M)\geq0\Rightarrow V_p(R)\leq \omega_dR^d\qquad\forall\, R\geq0,\] $\omega_d$ being the volume of the unit ball in $\R^d$.
\end{ex}

\begin{ex}  If $M$ is simply connected with all sectional curvatures $K\leq k$ for some $k<0$ and $Ric(M)\geq-\kappa^2>-\infty$ then $hR^\alpha$ exists over $M$ for any $\alpha>0$.  For example this holds if $M=\mathbb H^d$, $d$-dimensional hyperbolic space.  This follows from \cite{MR0266100} in which it is shown that $\sigma(-\Delta))\subset[(d-1)^2\frac{|k|}{4},\infty)$, which in turn implies the following upper bound on $H_t$ (see \cite{ MR1209255}):\[H_t(x,y)\leq C e^{\frac{(d-1)^2kt}{4}}\qquad\forall\, t\geq1\] for some $C>0$ and all $x,y\in M$.  Theorem 3.4 then applies.

In particular for $M=\mathbb H^2$, letting $\rho=d(x,y)$ we have the well known formula \[H_t(x,y)=\frac{\sqrt2}{(4\pi t)^\frac32}e^{-\frac 14 t}\int_{\rho}^\infty\frac{se^{-\frac {s^2}{4t}}}{\cosh(s)-\cosh(\rho)}\,ds.\] Then \[\E[hR^\alpha_xhR^\alpha_y]=\frac{\sqrt2}{(4\pi)^\frac 32\Gamma\left(1+\alpha\right)}\int_0^\infty \int_{\rho}^\infty t^{\alpha-\frac 32}\frac{se^{-\frac {1+s^2}{4t}}}{\cosh(s)-\cosh(\rho)}\,ds\,dt.\]

\end{ex}

\begin{rmk}On negatively curved manifolds, $hR^\alpha$ can also be viewed as an extension of the $fBf^\alpha$ in the following way:  In Section 3.1 we saw how the covariance of the $fBf^\alpha$ is the integral kernel of the operator $A^*A$ where \[A(f)=(-\Delta)^{-(\frac d4+\frac\alpha2)}(f)(x)-(-\Delta)^{-(\frac d4+\frac\alpha2)}(f)(0),\] which can be seen as a correction to $(-\Delta)^{-(\frac d4+\frac\alpha2)}$ when this operator does not have an integral kernel.  However on manifolds with spectrum as in example 3.5 $(-\Delta)^{-(\frac d4+\frac\alpha2)}$ does have an integral kernel and no correction is needed.  So if we view the $fBf^\alpha$ as the GRF with covariance that is the kernel of the minimal correction to $(-\Delta)^{-(\frac d4+\frac\alpha2)}$ that yields an integral operator, then on such manifolds as above we obtain the $hR^\alpha$.  
\end{rmk}
\subsection{H\"{o}lder Regularity}   

Having done the analytical work to build our covariances and check when they exist, we now turn to verifying that these covariances do in fact define random fields with the desired properties.  The first of those properties is in some ways the most fundamental:  Do the corresponding GRF's define probability measures on nice function spaces?  What we shall see is that if $M$ is compact or a compact subset of a regular domain or non-compact manifold over which the Riesz fields exist, then with probability one they have continuous sample paths and thus they determine probability measures on $C(M)$.

If $M$ is any Riemannian manifold or regular domain with heat kernel $H_t(x,y)$ then the maximum principle implies \[H_t(x,y)\leq H_t(x,x) \qquad\forall\, x,y\in M\] with equality if and only if $y=x$.  We then have that \[H_t(x,x)-2H_t(x,y)+H_t(y,y)>0\qquad \forall\, y\neq x. \]  In particular $\E[|R^\alpha_x-R^\alpha_y|^2]$ and $\E[|hR^\alpha_x-hR^\alpha_y|^2]$ both define metrics on $M$ when they exist.

Note also that \[\E[|R^\alpha_x-R^\alpha_y|^2]=\E[|hR^\alpha_x-hR^\alpha_y|^2]\] when both exist.  In particular in the proof below we will not distinguish these two metrics as the context of the Theorem will make clear which is being discussed.

We are now in a position to prove the following:
\begin{thm}Let $M$ be a compact Riemannian manifold, a regular domain, or non-compact under the hypothesis of Theorem 3.4.  We then have the following:
\begin{enumerate}
\item If $M$ is compact then there exists a version, $\tilde R^\alpha$, of $R^\alpha$ such that with probability $1$ the sample paths of $\tilde R^\alpha$ are uniformly H\"older continuous of any order $\gamma<\alpha$ on $M$, and there exists a dense subset of $M$ such that with probability $1$ the sample paths of $\tilde R^\alpha$ fail to be H\"older continuous at these points for any $\gamma>\alpha$.
\item If $M$ is a regular domain or non-compact under the hypothesis of Theorem 3.4, then for any compact set $K\subset M$ there exists a version, $\tilde R^\alpha$, of $R^\alpha$ such that with probability $1$ the sample paths of $\tilde R^\alpha$ are uniformly H\"older continuous of any order $\gamma<\alpha$ on $K$, and there exists a dense subset of $K$ such that with probability $1$ the sample paths of $\tilde R^\alpha$ fail to be H\"older continuous at these points for any $\gamma>\alpha$.

\end{enumerate}
\end{thm}

\begin{proof}In order to apply Theorem A.4 in the appendix we need to compare the metric $\E[|R^\alpha_x-R^\alpha_y|^2]$ (resp$.$ $\E[|hR^\alpha_x-hR^\alpha_y|^2]$) on $(M,g)$ with the metric $d(x,y)$ derived from $g$, in particular study the boundedness of \be \frac{\E[|R^\alpha_x-R^\alpha_y|^2]}{(d(x,y))^{2\gamma}}\ee for $d(x,y)$ small and $\gamma\in(0,1)$.  What we will show is that this ratio is unbounded if $\gamma>\alpha$ and approaches zero if $\gamma<\alpha$.  

Our approach to controlling (3.7) will be to split the integral defining $\E[|R^\alpha_x-R^\alpha_y|^2]$ into two parts: \begin{align}\notag \int_0^\infty t^{\frac d2+\alpha-1}&(H_t(x,x)-2H_t(x,y)+H_t(y,y))\,dt\\&=\int_0^1 t^{\frac d2+\alpha-1}(H_t(x,x)-2H_t(x,y)+H_t(y,y))\,dt\,\\&\quad+\int_1^\infty t^{\frac d2+\alpha-1}(H_t(x,x)-2H_t(x,y)+H_t(y,y))\,dt.\end{align} 

We start with (3.8).  Recall that in any case around any point $p$ there is a closed disk $D_p$  such that (2.1) holds with $\Phi\in C^k(\overline{D_p}\times\overline{ D_p}\times[0,T])$ where we can choose $k>2$ and $T>0$.  

As a consequence we have, denoting the integral (3.8) by $I_1$ and $d(x,y)$ by $\rho$,  \be I_1=(4\pi)^{-\frac d2}\int_0^1 t^{\alpha-1}(\Phi(t,x,x)+\Phi(t,y,y)-2\Phi(t,x,y)e^{-\frac {\rho^2}{4t}})\,dt.\ee  Because $\Phi\in C^k(\overline{D_p}\times\overline{ D_p}\times[0,T])$ with $k>2$ and is symmetric, by Lemma A.1 in the appendix, \[\Phi(t,x,x)+\Phi(t,y,y)-2\Phi(t,x,y)=O(\rho^2)\qquad\mbox{as } \rho\to0 \] uniformly for $t\in[0,1]$.  Thus we have \begin{align}\notag&\int_0^1 t^{\alpha-1}(\Phi(t,x,x)+\Phi(t,y,y)-2\Phi(t,x,y)e^{-\frac {\rho^2}{4t}})\,dt\\\notag&\quad=2\int_0^1t^{\alpha-1}\Phi(t,x,y)(1-e^{-\frac{\rho^2}{4t}})\,dt\\\notag&\qquad+\int_0^1t^{\alpha-1}(\Phi(t,x,x)+\Phi(t,y,y)-2\Phi(t,x,y))\,dt\\\notag&\quad   =2\int_0^1t^{\alpha-1}\Phi(t,x,y)(1-e^{-\frac{\rho^2}{4t}})\,dt+O(\rho^2)\end{align}

Because \[\varlimsup_{x\to y}\int_0^1t^{\alpha-1}\Phi(t,x,y)(1-e^{-\frac{\rho^2}{4t}})\,dt=\varlimsup_{x\to y}\rho^{2\alpha}\int_0^{\rho^{-2}}t^{\alpha-1}\Phi(\rho^2t,x,y)(1-e^{-\frac{1}{4t}})\,dt\] and \[\varlimsup_{x\to y}\int_0^{\rho^{-2}}t^{\alpha-1}\Phi(\rho^2t,x,y)(1-e^{-\frac{1}{4t}})\,dt<\infty,\]  \be{I_1}=O({\rho^{2\alpha}})=O({d(x,y)^{2\alpha}})\qquad\mbox{as }d(x,y)\to0\ee for $x,y\in {D_p}$.

For (3.9), which we denote $I_2$, we first deal with the case of $M$ compact. Using (2.2) we have for $t\geq1$ \begin{align}\notag H_t(x,x)-2H_t(x,y)+H_t(y,y)&=\sum_{k=0}^\infty e^{-\lambda_kt}|\phi_k(x)-\phi_k(y)|^2\\\notag&=\sum_{k=1}^\infty e^{-\lambda_kt}|\phi_k(x)-\phi_k(y)|^2\\\notag&\leq d(x,y)^2\sum_{k=1}^\infty e^{-\lambda_kt}\|\nabla\phi_k\|_\infty.\end{align}

Now we apply the following bound on $\|\nabla\phi_k\|_\infty$ (see \cite{MR2657840}): \[ \|\nabla\phi_k\|_\infty\leq C_M\lambda_k^{\frac {d+1}4}\] where $C_M$ is a constant depending only on $M$.  We then have \[H_t(x,x)-2H_t(x,y)+H_t(y,y)\leq C_Md(x,y)^2\sum_{k=1}^{\infty}e^{-\lambda_kt}\lambda_k^{\frac{d+1}4}=C_M d(x,y)^2O\left(e^{-\lambda_1 t}\right),\] which yields \be I_2\leq C_Md(x,y)^2\int_1^\infty t^{\frac d2+\alpha-1} O\left(e^{-\lambda_1 t}\right)\,dt=Cd(x,y)^2\ee as $\lambda_1>0$.

If $M$ is a regular domain then a similar argument using the corresponding bound (see \cite{MR2526794}) \[\|\nabla\phi_k\|_\infty\leq C_M\lambda_k^{\frac {d+1}4}\] for the Dirichlet eigenfunctions on $M$ we obtain (3.13) in this case as well.  Thus for either $M$ compact or a regular domain \[I_2=O\left(d(x,y)^2\right)\qquad \mbox{as } d(x,y)\to0.\]

Turning now to the case of $M$ non-compact, first suppose the first hypothesis of Theorem 3.4 is in force.  As in that proof we have, for $x,y$ contained in a sufficiently small geodesic disc, \begin{align}\notag\int_1^\infty t^{\frac d2+\alpha-1}&\left(H_t(x,x)-2H_t(x,y)+H_t(y,y)\right)\,dt\\\notag= &\int_1^\infty t^{\frac d2+\alpha-1}\int_M|H_t(x,z)-H_t(y,z)|^2\,dV_g(z)\,dt\\\notag\leq &\,d(x,y)^2\int_1^\infty t^{\frac d2+\alpha-1}\int_M|\nabla_xH_t(\xi_z,z)|^2\,dV_g(z)\,dt,\end{align} which was shown to be finite.

Next suppose the second hypothesis holds.  For this case we will use a Schauder estimate and Lemma A.1:  We choose a geodesic disc $D_p$ and let $ L$ be $\Delta$ in geodesic normal coordinates on $D_p$, $\mathbf{D}=exp^{-1}(D_p)$, $P=\partial_t-L$ on $C^\infty(\mathbf{D}\times(0,1))$, and $u(x',y',t)\in C^\infty(\mathbf{D}\times\mathbf{D}\times(0,1))$ be $H_t(x,y)$ in our chosen coordinates.  For any $T>0$ we then have \[Pu(x',y',t+T)=\partial_tu(x',y',t+T)-L_{x'}u(x',y',t+T)=0\] for each for all $x',y',t\in\mathbf{D}\times \mathbf{D}\times(0,1/2)$.  In other words, $u$ satisfies $Pu(x',y',t)=0$ on $\mathbf D\times (T,T+1/2)$ for each $y'\in\mathbf D$ and $T>0$.

Because $L$ is uniformly elliptic on $\mathbf{D}$ and its coefficients are all $C^\infty$ (and independent of $T$, $t$), using the Schauder estimate (Theorem 5 p.64 in \cite{MR0181836} and choosing $\alpha=1$) we obtain for each closed disk $\mathbf D_r$ contained in $\mathbf D$ a constant $K_r>0$ such that \[\sup_{(x',t)\in \mathbf D_r\times(0,1/2)}\left|\frac{\partial^2 u}{\partial x'_ix'_j}(x',y',t+T)\right|\leq K_r\sup_{(x',t)\in\mathbf {D}_r\times(0,1/2)}|u(x',y',t+T)|\] for each $i,j$ and $y'\in D_r$.  We then have \[\sup_{(x',y',t)\in\mathbf {D}_r\times\mathbf {D}_r\times(0,1/2)}\left|\frac{\partial^2 u}{\partial x'_ix'_j}(x',y',t+T)\right|\leq K_r\sup_{(x',y',t)\in\mathbf {D}_r\times\mathbf {D}_r\times(0,1/2)}|u(x',y',t+T)|.\]  We note that $K_r$ is independent of $T$ and by our hypothesis\hfill\\ $\sup_{(x,y)\in D_p\times D_p}H_t(x,y)\leq C t^{-(\frac d2+\beta)}$, $\beta>0$.  Thus, returning to $D_r=exp\,(\mathbf {D}_r)$, for all $T>1$ \[\sup_{(x,y,t)\in {D}_r\times {D}_r\times(0,1/2)}\left|\frac{\partial^2 H}{\partial x_ix_j}(x,y,t+T)\right|\leq CK_r T^{-(\frac d2+\beta)}.\]

Then applying Lemma A.1 and assuming without loss of generality we have chosen our disc $D_p$ such that the above estimates hold, we have

\begin{align}\notag&\int_1^\infty t^{\frac d2+\alpha-1}\left(H_t(x,x)-2H_t(x,y)+H_t(y,y)\right)\,dt\\\notag&\qquad\leq Cd(x,y)^2\int_1^\infty t^{\frac d2+\alpha-1}\sup_{\overline D_p\times \overline D_p}\left|\sum_{i,j=1}^d\frac{\partial^2 H}{\partial x_i\partial x_j}(t,x,y)\right|dt\\\notag&\qquad \leq Cd(x,y)^2\int_1^\infty t^{\frac d2+\alpha-1}(t-1/2)^{-(\frac d2+\beta)}\,dt\end{align}  
 for some $C>0$.
 By hypothesis $\beta>0$, so $\int_1^\infty t^{\frac d2+\alpha-1}(t-1/2)^{-(\frac d2+\beta)}\,dt<\infty$.  Lastly recall that when $hR^\alpha$ exists for $\alpha\in(0,b)$ for some $b>0$ then $R^\alpha$ does as well.  Moreover in that case \[\E[|R^\alpha_x-R^\alpha_y|^2]=\E[|hR^\alpha_x-hR^\alpha_y|^2],\] so in the second case of Theorem 3.4 the arguments above apply to $R^\alpha$ as well.
 
Thus in each case from the preceeding discussion we know that for each $p\in M$ there exists a closed disc $D_p$ centered at $p$ such that for all $\gamma\leq\alpha$ \[\E[|R^\alpha_x-R^\alpha_y|^2]\leq C_p (d(x,y)^{2\gamma})\] for some constant $C_p>0$ and all $x,y\in D_p$ and that such a condition fails for any $\gamma>\alpha$ in light of (3.11).  Then if $M$ is compact or $K$ is a compact subset of $M$, there exists a constant $C>0$ such that for all $\gamma\leq\alpha$ \[\E[|R^\alpha_x-R^\alpha_y|^2]\leq C d(x,y)^{2\gamma}\] for all $x,y\in M$ (resp$.$ $x,y\in K$).  Then by Theorem A.4 in the appendix there is a version of $R^\alpha$ that is almost surely uniformly H\"{o}lder continuous over $M$ (resp$.$ $K$) of order $\gamma$ for any $\gamma<\alpha$.  Moreover from the discussion following Theorem A.4 there is a dense subset of $M$ (resp$.$ $K$) on which $R^\alpha$ fails to satisfy any H\"{o}lder condition of order $\gamma$ for any $\gamma>\alpha$ with probability 1.  By the remarks preceding the Theorem the same holds for $hR^\alpha$, when it exists.
\end{proof}
\begin{rmk}  From the proof above we see that \[\lim_{x\to y}\frac{\E[|R^\alpha_x-R^\alpha_y|^2]}{d(x,y)^{2\alpha}}=\lim_{x\to y}\int_0^{d(x,y)^{-2}}t^{\alpha-1}\Phi(d(x,y)^2t,x,y)(1-e^{-\frac{1}{4t}})\,dt\] and thus the exact comparison between the Riemannian metric of $M$ and the metric induced by $R^\alpha$ depends on the local geometry of $M$, in particular on the comparison with the Euclidean heat kernel contained in $\Phi(t,x,y).$

\end{rmk}

\begin{rmk}It would be desirable in the case of regular domains to extend continuity to the closure of $M$.  However the local Euclidean approximation of the heat kernel is not uniform near the boundary of $M$ and so some other method of proof seems necessary. On the other hand it is easy to show that for any sequence $(x^k_1,\dots,x^k_n)$ that approaches the boundary of $M$, $P(\|(hR^\alpha_{x^k_1},\dots,hR^\alpha_{x^k_n})\|>\epsilon)\stackrel{k}\to0$ for any $\epsilon>0$.  This combined with the existence of a continuous version as close as we like to the boundary seems sufficient for most applications, at least from the point of view of simulation.
\end{rmk}
\subsection{Distributional Scaling and Invariance}  
\subsubsection{Stationarity}  

\begin{defn}Let $(M,g)$ be a complete Riemannian manifold and $I(M)$ the group of isometries of $(M,g)$.  If $Y_x$ is a centered GRF over $(M,g)$ we say that $Y_x$ is \textit{stationary (or homogeneous)} if \[\E[Y_{\iota(x)}Y_{\iota(y)}]=\E[Y_xY_y]\] for any $\iota\in I(M)$ and all $x,y\in M$.  We say $Y_x$ has \textit{stationary (or homogeneous) increments} if \[\E[|Y_{\iota(x)}-Y_{\iota(y)}|^2]=\E[|Y_x-Y_y|^2]\] for any $\iota\in I(M)$ and all $x,y\in M$.
\end{defn}

Because for any manifold $(M,g)$ we have $H_t(\iota(x),\iota(y))=H_t(x,y)$ for any $\iota\in I(M)$ (see \cite{MR2569498}, Theorem 9.12) it is clear from the definitions,
\[ \E[R^\alpha_xR^\alpha_y]=\frac1{\Gamma\left(\frac d2+\alpha\right)}\int_0^\infty t^{\frac d2+\alpha-1}\left(H_t(x,y)-H_t(x,o)-H_t(y,o)+H_t(o,o)\right)\,dt\] and 
\[\E[hR^\alpha_xhR^\alpha_y]=\frac1{\Gamma\left(\frac d2+\alpha\right)}\int_0^\infty t^{\frac d2+\alpha-1}H_t(x,y)\,dt,\] that when they exist, $R^\alpha$ and $hR^\alpha$ have stationary increments and are stationary respectively.
\subsubsection{Self-Similarity}

Turning to self-similarity, let us first recall how this property is defined for random fields on Euclidean space: If $Y_x$ is a random field over $\mathbb R^d$, then $Y_x$ is \textit{self-similar} of order $\alpha>0$ if $c^{\alpha}Y_{\frac1cx}\stackrel{d}{=}Y_x$.  The Euclidean fractional Brownian field $fBf^\alpha$ is self similar of order $\alpha$, and we want to extend this property to manifolds.  To do this we must define an operation that extends the scaling operation on $\mathbb R^d$, $x\mapsto cx$.  This operation scales the distance between any two points by $c>0$: \[\|x-y\|\mapsto\|cx-cy\|=c\|x-y\|,\] or written another way, \[d(x,y)\mapsto cd(x,y).\]  Viewing $\mathbb R^d$ as a manifold, we see this is equivalent to scaling the Riemannian metric $(g_{ij})=(\delta_{ij})$ of $\mathbb R^d$ by $c^2$, \[\sum_{i,j=1}^dx_ix_jg_{ij}=\sum_{i=1}^dx_i^2=\|x\|^2\mapsto c^2\|x\|^2=\sum_{i,j=1}^dx_ix_jc^2g_{ij}.\]

Thus a natural definition of scaling for a manifold $M$ is to simply scale the metric as above.  Indeed, if $M$ is an embedded submanifold of $\mathbb R^d$ with induced metric $g_M$, then scaling the ambient space $\mathbb R^d$ results in the induced scaling on $M$ \[g_M\mapsto c^2g_M.\]  Of course, we'd like a definition of scaling that is intrinsic to the manifold in question, i.e., independent of any ambient Euclidean space, but that also agrees with the scaling induced by scaling any ambient space.  If we take the above operation as the definition of scaling for a general manifold $M$ we achieve this goal.  

We are thus ready to prove that the Riesz fields are self-similar.

\begin{prop}  Let $(M,g)$ be a complete Riemannian manifold or regular domain.  Both the Riesz field $R^\alpha$ and the stationary Riesz field $hR^\alpha$ over $(M,g)$ are self-similar of order $\alpha$ (if they exist on $M$) in the sense that if $\bar{R}^\alpha$ and $h\bar{R}^\alpha$ are the Riesz fields over $(M,c^2g)$  then \[c^\alpha R^\alpha_{x}\stackrel{d}{=}\bar{R}^\alpha_x\] and \[c^\alpha hR^\alpha_{x}\stackrel{d}{=}h\bar{R}^\alpha_x\] for any $c>0$.

\end{prop}

\begin{proof}  First we note from the coordinate expression for $\Delta$, if we denote by $\Delta_g$ the Laplacian of $(M,g)$ and $H^g_t(x,y)$ the corresponding heat kernel, we have $\Delta_{c^2g}=\frac 1{c^2}\Delta_g$.  But then because $L^2(M,dV_g)=L^2(M,dV_{c^2}g)$ we can write \begin{align}\notag\int_M c^dH^{c^2g}_t(x,y)f(y)\,dV_g(y)&=\int_M H^{c^2g}_t(x,y)f(y)\,dV_{c^2g}(y)\\\notag&=e^{-t\Delta_{c^2g}}(f)\\\notag&=e^{-\frac{t}{c^2}\Delta_{g}}(f)\\\notag&=\int_M H^{g}_{\frac{t}{c^2}}(x,y)f(y)\,dV_{g}(y)
\end{align} 
for any $f\in L^2(M,dV_g)$.  Thus by symmetry \[\frac1{c^d}H^g_{\frac t{c^2}}(x,y)=H^{c^2g}_t(x,y)\qquad\forall\, x,y\in M.\]  We then have \begin{align}c^{2\alpha}\notag\E[R^\alpha_{x}&R^\alpha_{y}]=\frac{c^{2\alpha}}{\Gamma\left(\frac d2+\alpha\right)}\int_0^\infty t^{\frac d2+\alpha-1}\left(H^{g}_t(x,y)-H^{g}_t(o,x)-H^{g}_t(o,y)+H^{g}_t(o,o)\right)\,dt\\\notag&=\frac1{\Gamma\left(\frac d2+\alpha\right)}\int_0^\infty t^{\frac d2+\alpha-1}\frac1{c^d}\left(H^{g}_{\frac{t}{c^2}}(x,y)-H^{g}_{\frac{t}{c^2}}(o,x)-H^{g}_{\frac{t}{c^2}}(o,y)+H^{g}_{\frac{t}{c^2}}(o,o)\right)\,dt\\\notag&=\frac1{\Gamma\left(\frac d2+\alpha\right)}\int_0^\infty t^{\frac d2+\alpha-1}\left(H^{c^2g}_t(x,y)-H^{c^2g}_t(o,x)-H^{c^2g}_t(o,y)+H^{c^2g}_t(o,o)\right)\,dt\\\notag&=\E[\bar{R}^\alpha_{x}\bar{R}^\alpha_{y}]\end{align} and similarly for $hR^\alpha$.

\end{proof}
\begin{rmk}  Here we see that $hR^\alpha$ exhibits essentially non-Euclidean phenomena; on $\R^d$ there cannot exist a GRF that is both stationary and self similar (see e.g$.$ \cite{MR1364671}).  We will return to the questions this raises in Section 5. \end{rmk}
\subsubsection{Uniqueness}
We now come to a natural question:  Are the Riesz fields the only fields with stationary increments that are also self-similar?  In other words, does requiring stationarity and self-similarity as above uniquely determine a GRF over a given manifold $M$?  To answer this we examine an example, $M=\mathbb S^1$, which we normalize to have total volume $2\pi$.  Using the expansion of section 3.2.1 we have \[R^\alpha(x)\stackrel d= \sum_{k\in\mathbb Z\backslash\{0\}}\frac1{\sqrt{2\pi}}|k|^{-\frac12-\alpha}(e^{ikx}-1)\xi_k.\]  In \cite{MR2198600} the author constructs a GRF, denoted $R_\alpha$, with the following covariance \[\frac12(d(x,0)^{2\alpha}+d(y,0)^{2\alpha}-d(x,y)^{2\alpha}).\]  In particular it is shown that \[R_\alpha(x)\stackrel{d}=\sum_{k\in\mathbb Z\backslash\{0\}} d_k(e^{ikx}-1)\xi_k\] where  \begin{align}\notag d_k&=\frac{\sqrt{-\int_0^{|k|\pi}u^{2\alpha}\cos(u)du}}{\sqrt{2\pi}|k|^{\frac12+\alpha}}.\end{align}Note however that for $\alpha=\frac12$, \[d_k=\begin{cases}0\qquad &k\,\mbox{even},\\(\sqrt{\pi}|k|)^{-1}\qquad &k\,\mbox{odd}\end{cases}.\]  Thus \[R_\frac12(x)=\sum_{k\in\mathbb Z\backslash\{0\}}\frac1{\sqrt\pi}|2k+1|^{-1}(e^{i(2k+1)x}-1)\xi_k\] and \[\sqrt{2}R^\frac12(x)=\sum_{k\in\mathbb Z\backslash\{0\}}\frac1{\sqrt{\pi}}|k|^{-1}(e^{ikx}-1)\xi_k.\]  We then find that \[\mathbb E[|\sqrt{2}R^\frac12(x)|^2]-\mathbb E[|R_\frac12(x)|^2]=\sum_{k=-\infty}^{\infty}\frac1{{\pi}}|2k|^{-2}|e^{i2kx}-1|^2,\] which is not identically zero.  As their variances are not identical, these two fields are not equal in distribution.  However it is easy to see that both fields have stationary increments and are self-similar of order $1/2$.  

Thus even in the simple case of $\mathbb S^1$ we do not have uniqueness, and so in general the Riesz fields are not the only GRF's that are self-similar with stationary increments over a given manifold $M$.  It then remains an open question to determine the general form of the covariance of a GRF with stationary increments that is also self-similar over a given manifold other than $\R^d$.

\section{The Bessel Field}  We now turn to constructing stationary counterparts to $R^\alpha$ by analogy with the Brownian motion and Ornstein-Uhlenbeck processes on $\R$.  We define the \textit{Bessel Field} of order $\alpha\in (0,1)$ by \be B^\alpha_x\stackrel{d}{=}\frac1{\Gamma\left(\frac d4+\frac\alpha2\right)}\int_M\int_0^\infty t^{\frac d4+\frac\alpha2-1}e^{-t}H_t(x,z)\,dt\,dW(z),\ee which extends the Ornstien-Uhlenbeck fields with covariance given (up to a constant) by \[\int_{\R^d}\frac{e^{i\<x,y\>}}{(1+|\xi|^2)^{\frac d2+\alpha}}\,d\xi.\]  These fields are altogether more well behaved than the Riesz fields, which is not surprising in light of the analogy with the Riesz and Bessel potentials.

\begin{thm}  The Bessel field exists over any complete Riemannian manifold or regular domain $M$ for all $\alpha\in(0,1)$.\end{thm}
\begin{proof}Proceeding as for $hR^\alpha$, for each $x,y\in M$ \begin{align}\notag\E[B^\alpha_xB^\alpha_y]&=\left(\frac1{\Gamma\left(\frac d4+\frac\alpha2\right)}\right)^2\int_M\int_0^\infty t^{\frac d4+\frac\alpha2-1}e^{-t}H_t(x,z)\,dt\int_0^\infty s^{\frac d4+\frac\alpha2-1}e^{-s}H_s(y,z)\,ds\,dV_g(z)\\\notag&=\left(\frac1{\Gamma\left(\frac d4+\frac\alpha2\right)}\right)^2\int_0^\infty \int_0^\infty t^{\frac d4+\frac\alpha2-1} s^{\frac d4+\frac\alpha2-1}e^{-(t+s)}H_{t+s}(x,y)\,dt\,ds\\&=\frac1{\Gamma\left(\frac d2+\alpha\right)}\int_0^\infty t^{\frac d2+\alpha-1}e^{-t}H_t(x,y)\,dt
 \end{align}
From the fact that the heat kernel always satisfies $\varlimsup_{t\to\infty}H_t(x,y)<\infty$ for any $x$ and $y$, we see that (4.2) converges everywhere on $M\times M$.

\end{proof}

Clearly $B^\alpha_x$ is stationary and we can see that it does not possess the scaling properties of the Riesz fields.  Turning to sample path regularity we have the following result.

\begin{thm}The Bessel field $B^\alpha$ has a version with sample paths almost surely uniformly H\"{o}lder continuous of order $\gamma$ for any $\gamma<\alpha$ and almost surely failing to satisfy a H\"{o}lder condition of order $\gamma$ for any $\gamma>\alpha$ on a dense subset of $M$.
\end{thm}
\begin{proof}Split the integral 
\begin{align}\E[|B^\alpha_x-B^\alpha_y|^2]\notag&=  \frac1{\Gamma\left(\frac d2+\alpha\right)}\int_0^\infty t^{\frac d2+\alpha2-1}e^{-t}\left(H_t(x,x)-2H_t(x,y)+H_t(y,y)\right)\,dt\\\notag&=\frac1{\Gamma\left(\frac d2+\alpha\right)}(I_1+I_2)\end{align} where
\[I_1=\int_0^1 t^{\frac d2+\alpha2-1}e^{-t}\left(H_t(x,x)-2H_t(x,y)+H_t(y,y)\right)\,dt\] and \[I_2=\int_1^\infty t^{\frac d2+\alpha2-1}e^{-t}\left(H_t(x,x)-2H_t(x,y)+H_t(y,y)\right)\,dt\]

and argue as in Theorem 3.7.

\end{proof}
\section{Conclusion and Further Work}\subsection{Existence and Uniqueness}  Using a spectral theoretic approach we have constructed analogues of the fractional Brownian fields over arbitrary compact manifolds and a wide class of non-compact manifolds.  There are still many questions remaining. For example in light of the non-uniqueness result in Section 3.4.3, one could ask how many different such fields there are over any given manifold.  One could also attempt to determine the general form the covariance of such objects must take.

We also saw in example 3.3 that $R^\alpha$ does not exist on $\mathbb S^1\times\R$ (with the product metric) for $\alpha>1/2$.  This raises the following question:  Does there exist any Gaussian field over $\mathbb S^1\times\R$ with stationary increments that is also self similar of order $\alpha$ for some $\alpha\in (1/2,1)$?  More generally, are there geometric conditions that ensure a given manifold can have such a field defined over it?  

We conjecture that it is possible to construct such fields over any manifold $M$ in the following way:  Somewhat informally, the Riesz fields are solutions to the stochastic equation \[(-\Delta)^{\frac d4+\frac\alpha2}X=W,\] where $W$ is Gaussian white noise over $M$ and $\Delta$ is the Laplacian of $M$ with certain ``boundary conditions," i.e., with domain restricted to include only functions $f$ such that $f(o)=0$ for some fixed point $o\in M$.  As we saw, for example in the case of compact manifolds, this restriction of the domain led to the existence of a continuous integral kernel for the corresponding inverse and it seems plausible that in general we could always obtain such a kernel through restricting the domain of $\Delta$ by determining a sufficient number of derivatives of $f\in Dom(\Delta)$ at the point $o$.  Of course finding an explicit expression for such a kernel may be very difficult in general.
\subsection{Restriction to Submanifolds}
There is one aspect of this theory we did not touch upon, that being the behavior of our fields when restricted to geodesics and more general submanifolds.  One thing we can say is that for a given manifold $M$, following the discussion of self-similarity and dilation in Section 3.7, the Riesz fields over $M$ when restricted to an embedded submanifold $N$ determine self-similar fields over $N$.  Also, being embedded, the isometry group of $N$ determines a (possibly trivial) subgroup of the general isometry group of the $M$.  However, the resulting restricted field may be stationary or have stationary increments (for example, consider the $fBf^\alpha$ over $\R^d$ restricted to $\mathbb S^{d-1}$).  Moreover, as we already saw, stationarity and self-similarity alone do not uniquely determine a GRF in general, and so we cannot say that $R^\alpha$ over $M$ when restricted to a submanifold $N$ agrees with $R^\alpha$ over $N$.

While we have avoided symmetry hypothesis in our treatment, when dealing with invariance properties involving isometry groups one is naturally led towards general harmonic analysis and it would be interesting to study GRF's over manifolds from this point of view.  For example, one could consider GRF's that are only stationary with respect to a subgroup of the entire isometry group, analogous to GRF's over $\R^d$ that are only rotationally invariant (so called \textit{isotropic} random fields).

One property of the Euclidean fractional Brownian fields (or more generally any GRF that is self-similar with translation invariant increments) is that when restricted to lines through the origin they agree with the usual fractional Brownian motion, up to a constant.  One could then ask if this holds more generally.  For example one could require that a field over $M$ when restricted to infinite geodesics became a fractional Brownian motion.  This would require a subgroup of the isometry group of $M$ that restricted to translation of the given geodesic.  Of course, in general geodesics may be closed or infinite.  Again, one could study such questions from a general harmonic analytic point of view.

\subsection{Hyperbolic GRF's}
We also mentioned above that the existence of $hR^\alpha$ raises interesting questions regarding negatively curved manifolds and what we could loosely call hyperbolic Gaussian random fields.  For example, although the proof of existence of $hR^\alpha$ over $\mathbb H^d$ uses properties of the heat kernel, one can ask if there are more geometric or topological conditions one can put on a manifold $M$ to ensure the existence of some self-similar and stationary GRF.  Conversely one can ask what are the implications of such a field existing over $M$.  Is $hR^\alpha$ the only such field or are there others?

The above is only a first attempt to state some questions at the intersection of geometry and probability that, at least on the face of it, seem novel and interesting; doubtless there are others.  The study of random fields over manifolds, although its history is not short, seems to the author to still be wide open.  It is our hope that the work here and the questions raised above will be of interest to both researchers in geometry or geometric analysis and probabilists and lead to further interaction between the two.

\section{Acknowledgements}
The author wishes to express deep gratitude to his advisor, Harold Parks, for his confidence,  generosity, and encouragement.  The author also thanks his teachers David Finch, Mina Ossiander, and Ralph Showalter for their time and encouragement as well as Robert Higdon and Kirk Lancaster for helpful comments and suggestions.  Lastly the author thanks the anonymous referee for their valuable remarks. 

\appendix \section{}
First we record the following Lemma involving Taylor approximation.
\begin{lem}Let $M$ be complete and suppose $f\in C^\infty(M\times M)$ is symmetric.  Around any point $p \in M$ there exists a closed geodesic disk  $D_p$ centered at $p$ and a constant $C_p>0$ such that \[|f(x,x)-2f(x,y)+f(y,y)|\leq C_pd(x,y)^2\sup_{D_p\times D_p}\left|\sum_{i,j=1}^d\frac{\partial^2 f}{\partial x_i\partial x_j}\right|\] for all $x,y\in D_p$.

\end{lem}
\begin{proof}Let $F\in C^2(\R^d)$ and recall Taylor's Theorem: for each $p\in \R^d$ and all $x\in \R^d$\begin{align}\notag F(x)=F(p)&+\sum_{i=1}^d\frac {\partial F}{\partial x_i}(p)(x_i-p_i)\\\notag&+\sum_{i,j=1}^d(x_i-p_i)(x_j-p_j)\frac{2}{1+\delta_{ij}}\int_0^1(1-t)\frac{\partial^2 F}{\partial x_i\partial x_j} (p+t(x-p))dt.\end{align}  

Now let $f\in C^2(\R^d\times \R^d)$ and $f(x,y)=f(y,x)$.  Fix $x,y\in \R^d$.  Then letting $p=(x,y)$, from the symmetry of $f$ we have \begin{align}\notag f(x,x)-&2f(x,y)+f(y,y)\\\notag&\qquad
=\sum_{i,j=1}^d(x_i-y_i)(x_j-y_j)\int_0^1(1-t)\frac{\partial^2 f}{\partial x_i\partial x_j} (x+t(y-x),x)dt\\\notag&\qquad\quad+\sum_{i,j=1}^{d}(x_i-y_i)(x_j-y_j)\int_0^1(1-t)\frac{\partial^2 f}{\partial x_{i}\partial x_{j}} (y+t(x-y),y)dt\\\notag&\qquad
=\int_0^1\sum_{i,j=1}^d(x_i-y_i)(x_j-y_j)(1-t)\left(\frac{\partial^2 f}{\partial x_i\partial x_j} (x+t(y-x),x)\right.\\\notag&\qquad\quad+\left.\frac{\partial^2 f}{\partial x_{i}\partial x_{j}} (y+t(x-y),y)\right)dt
\\\notag&\qquad=c\sum_{i,j=1}^d(x_i-y_i)(x_j-y_j)\left(\frac{\partial^2 f}{\partial x_ix_j}(x+\theta_1,x)+\frac{\partial^2 f}{\partial x_ix_j}(y+\theta_2,y)\right)\end{align} for some constant $c>0$ and $\theta_k\in\R^{d}$ with $\|\theta_k\|_{\R^{d}}<\|x-y\|_{\R^d}$. 
In particular for $x,y$ in a closed disk $D_\epsilon$ of radius $\epsilon>0$ we have \[|f(x,x)-2f(x,y)+f(y,y)|\leq C_1\|x-y\|_{\R^d}^2\sup_{D_\epsilon\times D_\epsilon}\left|\sum_{i,j=1}^d\frac{\partial^2 f}{\partial x_i\partial x_j}\right|\] for some $C_1>0$.   

Now suppose $f\in C^\infty(M\times M)$ is symmetric and let $D_p$ be a geodesic disk centered at $p\in M$. Then the above implies \be|f(x,x)-2f(x,y)+f(y,y)|\leq C_2d(x,y)^2\sup_{\overline D_p\times \overline D_p}\left|\sum_{i,j=1}^d\frac{\partial^2 f}{\partial x_i\partial x_j}\right|\ee for all $x,y\in D_p$. 

\end{proof}
\subsection{Continuity of Gaussian random fields}

Here we provide analogues of results given for Gaussian fields over $\R^d$ in the setting of manifolds.  These proofs are simple modifications of the originals and we include them for convenience.  The first result is an analytical lemma, given for hypercubes in $\R^d$.  We will replace the cubes with metric disks and $\R^d$ by a $d$-dimensional manifold $M$.  Let $p$ be even and continuous on $[-1,1]$, $p(|x|)$ monotone increasing, and satisfy $\lim_{x\to0}p(x)=0$.

\begin{lem} (Manifold version of Lemma 1 in \cite{MR0410880}):  Let $f\in C(I_0)$ where $I_0\subset M$ is compact, has non-empty interior, and has no isolated points.  Suppose that \[\int_D\int_Dexp\left(\frac{f(x)-f(y)}{p(\mbox{diam}(D))}\right)^2\,dx\,dy\leq B\] for all closed metric disks $D\subset I_0$.  Then for some $C>0$ \[|f(x)-f(y)|\leq 8\int_0^{d(x,y)}\sqrt{\log(BCu^{-2d})}\,dp(u)\] for all $x,y\in I_0$.

\end{lem}
\begin{proof}  Fix $x,y\in I_0$.  Then choose a sequence of disks $D_k=\{z\in M\,:\,d(z,x)<r_k\}$ such that $D_k\subset I_0$, $2r_1\leq d(x,y)$, $r_k\to0$, and if $d_k=2r_k$ we have \[p(d_k)=\frac12p(d_{k-1}).\]  
Let $f_{D_k}=\mbox{Vol}(D_k)^{-1}\int_{D_k}f\,dV$.  We apply Jensen's inequality to obtain \begin{align}\notag &exp\left(\frac{f_{D_k}-f_{D_{k-1}}}{p(d_{k-1})}\right)^2\\\notag&\qquad\leq [\mbox{Vol}(D_k)\mbox{Vol}(D_{k-1})]^{-1}\int_{D_k}\int_{D_{k-1}} exp\left(\frac{f(x)-f(y)}{p(d_{k-1})}\right)^2 \,dV(x)\,dV(y)\\\notag&\qquad\leq B[\mbox{Vol}(D_k)\mbox{Vol}(D_{k-1})]^{-1}.\end{align}   
We then have \be|f_{D_k}-f_{D_{k-1}}|\leq p(d_{k-1})\sqrt{\log(B[\mbox{Vol}(D_k)\mbox{Vol}(D_{k-1})]^{-1})}\ee
By the definition of $D_k$ we have \[p(d_{k-1})=4[p(d_k)-p(d_{k+1})].\] Then because \[\mbox{Vol}(D_k)=O\left((d_k)^d\right)\qquad \mbox {as } k\to\infty,\] $\exists$ $C>0$ such that \[\mbox{Vol}(D_k)\geq C(d_k)^d\] so that $d_{k+1}\leq u\leq d_k\Rightarrow u^{-2d}\leq C[\mbox{Vol}(D_k)\mbox{Vol}(D_{k-1})]^{-1}$.  Then we can write (4.1) as \[|f_{D_k}-f_{D_{k-1}}|\leq4\int_{d_{k+1}}^{d_k}\sqrt{\log(BCu^{-2d})}\,dp(u).\]
Summing these and using continuity of $f$ we get\[|f(x)-f_{D_1}|=\varlimsup_{k\to\infty}|f_{D_k}-f_{D_1}|\leq 4\int_0^{d_2}\sqrt{\log(BCu^{-2d})}\,dp(u).\]  
Now $d_2<d(x,y)$ so if we need to we can replace $B$ by a larger bound to ensure the integrand is defined, and after doing so we have \[|f(x)-f_{D_1}|\leq 4\int_0^{d(x,y)}\sqrt{\log(BCu^{-2d})}\,dp(u).\]  The argument is symmetric in $x$ and $y$, so an application of the triangle inequality yields the conclusion.

\end{proof}

Suppose now we are given a (centered) Gaussian random field $X_x$ over $(M,g)$ and consider its restriction to a compact set $I_0$ as above.  Suppose further that the function $K(x,y)=\E[X_xX_y]$ is continuous on $I_0\times I_0$.  Then $K(x,y)$ determines a positive trace class integral operator on $L^2(I_0,dV_g)$ and by Mercer's theorem we have \[K(x,y)=\sum_{k=0}^\infty \lambda_k\phi_k(x)\phi_k(y)\] uniformly on $I_0\times I_0$, where $\lambda_k$ and $\phi_k$ are the eigenvalues and eigenfunctions of $K$ respectively. 

Let $p(u)=\sup\{\sqrt{\E[|X_x-X_y|^2]}:d(x,y)\leq |u|\}$ and $X_x^n=\sum_{k=0}^n\sqrt{\lambda_k}\phi_k(x)\theta_k$ where the $\theta_k$ are independent standard normal random variables.

We then have the following adaptation of Garsia's theorem to the manifold setting:
\begin{thm}(Manifold version of Theorem 1 in \cite{MR0410880}):  Suppose that for $x,y\in I_0$ as above \[\int_0^{\mbox{diam}(I_0)\wedge1}\sqrt{-\log(u)}\,dp(u)<\infty.\]  Then with probability 1 \[|X^m_x-X^m_y|\leq \frac18\int_0^{d(x,y)}\sqrt{\log(BCu^{-2d})}\,dp(u)\] where $C>0$ and \[\sup_m\int_{I_0}\int_{I_0}exp\frac14\left(\frac{X^m_x-X^m_y}{p(d(x,y))}\right)^2\,dV(x)\,dV(y)\leq B<\infty\] almost surely.
In particular the partial sums $X^m_x$ are almost-surely equicontinuous and uniformly convergent on $I_0$.
\end{thm}

\begin{proof}  Let \[P_n=exp\frac18\left(\frac{X^n_x-X^n_y}{p(d(x,y))}\right)^2=P_{n-1}exp\frac18\left(\frac{(Y^n(x,y))^2-2Y^n(x,y)(X^{n-1}_x-X^{n-1}(y))}{p(d(x,y))}\right)^2\] where $Y^k(x,y)=\sqrt{\lambda_k}(\phi_k(x)-\phi_k(y))\theta_k$.  Then by independence of the $\theta_k$ and Jensen's inequality for conditional expectation \begin{align}\notag&\E[P_{n+1}\left|P_n,...,P_1\right.]\\\notag&\qquad=P_n\left(\E\left[\left.exp\frac18\left(\frac{X^{n+1}_x-X^{n+1}_y}{p(d(x,y))}\right)^2\right|P_n,...,P_1\right]\right)\\\notag&\qquad\geq P_n exp\frac18\left(\E\left[\left.\left(\frac{(Y^{n+1}(x,y))^2-2Y^{n+1}(x,y)(X^{n-1}_x-X^{n-1}(y))}{p(d(x,y))}\right)\right|P_n,...,P_1\right]\right)^2\\\notag&\qquad=P_nexp\frac18\left(\E\left[\left.\left(\frac{(Y^{n+1}(x,y))^2}{p(d(x,y))}\right)\right|P_n,...,P_1\right]\right)^2\\\notag&\qquad\geq P_n\qquad a.s.
\end{align}

Thus $\{P_n\}$ is a submartingale.  Next note that $\E[P_n^2]\leq\sqrt2$, as \[\frac{X^n_x-X^n_y}{p(d(x,y))}\] is centered, Gaussian, and has variance less than or equal to one.  Then applying the classical submartingale inequalities we have \[\E[\max_{m\leq n}P_m^2]\leq4\E[P_n^2]\leq4\sqrt2.\]

Applying the Fubini-Tonelli theorem we then have \[\E\left(\int_{I_0}\int_{I_0}\max_{m\leq n}exp\frac14\left(\frac{X^n_x-X^n_y}{p(d(x,y))}\right)^2\,dV(x)\,dV(y)\right)\leq4\sqrt2\left(V(I_0)\right)^2.\]  Letting $n$ tend to infinity and applying monotone converge yields \[\E[B]\leq4\sqrt2\left(V(I_0)\right)^2<\infty.\]  

We then have that almost surely \[\int_{I_0}\int_{I_0}exp\frac14\left(\frac{X^n_x-X^n_y}{p(d(x,y))}\right)^2\,dV(x)\,dV(y)\leq B<\infty\qquad\forall\, n\] so that Lemma A.2 applies.

Lastly note that from \[\E\left[\sum_{k=0}^\infty\lambda_k\theta_k^2\right]=\sum_{k=0}^\infty\lambda_k=\int_{I_0}K(x,x)\,dV(x)<\infty\] we obtain with probability one \[\sum_{k=0}^\infty \lambda_k\theta_k^2<\infty,\]  which together with the conclusion of Lemma A.2 implies the almost sure uniform convergence of $\{X_x^n\}$ on $I_0$.

\end{proof}

As remarked in \cite{MR0410880} this result gives a sufficient condition for the existence of an almost surely continuous version of $X_x$.  The next result establishes H\"{o}lder continuity. 

\begin{thm}(Manifold version of Thm 8.3.2 in \cite{MR611857}):  Let the field $X$ over $I_0\subset M$ be as above and let $\gamma=\sup\{\beta:\E[|X_x-X_y|^2]=o(d(x,y)^{2\beta})\,\mbox{uniformly on }I_0\}$.  Then there exists a version of $X$ with sample paths that are almost surely uniformly H\"{o}lder continuous over $I_0$ of any order $\beta<\gamma$.

\end{thm}

\begin{proof}  Let $\rho=d(x,y)$.  First note that, with $p(u)$ as above, we have for any $L>0$\[\int_L^\infty p(e^{-x^2})\,dx\leq c_\epsilon\int_L^\infty e^{-(\gamma-\epsilon)x^2}\,dx<\infty\] for any $0<\epsilon<\gamma$ and some constant $c_\epsilon$.  But this is equivalent to \[\int_0^{\mbox{diam}(I_0)\wedge1}\sqrt{-\log(u)}\,dp(u)<\infty.\]  Thus by the previous result we have a version (which we also denote by $X$) for which \[|X_x-X_y|\leq Bp(\rho)+ C\int_0^{\rho}\sqrt{-\log(u)}\,dp(u)\qquad a.s.\] for some constant $C>0$ and some positive random variable $B$ almost surely finite.

Now for any $0<\epsilon<\gamma$ we have some constant $C_\epsilon>0$ such that $p(\rho)<C_\epsilon \rho^{\gamma-\epsilon}$, and similarly $\int_0^{\rho}\sqrt{-\log(u)}\,dp(u)<C^\prime_\epsilon\rho^{\gamma-\epsilon}$ for some $C^\prime_\epsilon>0$.  Thus, with probability 1, for each $\epsilon>0$ there is an almost surely finite positive random variable $A_\epsilon$ such that \[|X_x-X_y|\leq A_\epsilon d(x,y)^{\gamma-\epsilon}\qquad\forall\, x,y\in I_0.\]

\end{proof}

Note that we can also show under the hypotheses of the theorem that in any disk of positive radius in $I_0$ the sample paths of $X$ fail to be uniformly H\"{o}lder of any order greater than $\gamma$.  Indeed, \[\frac{X_x-X_y}{d(x,y)^{\gamma+\epsilon}}\] is a centered Gaussian random variable with variance $O(d(x,y)^{-\frac{\epsilon}2})$ and thus becomes almost surely unbounded as $x\to y$.   For example we can pick any countable dense subset of $I_0$ and modify $X$ on a set of measure zero to obtain the failure of H\"{o}lder continuity at each point in the set. Any stronger converse statement will require more refined tools, i.e., local times, which we will not attempt to develop here.

\begin{rmk} We mention here that the results in \cite{MR2604006}, of which the author became aware after submission of the present article, may be an alternative to the results above for establishing sample path continuity in Theorem 3.6.
\end{rmk}
\bibliographystyle{plain}
\bibliography{Thesisbib}
\end{document}